\documentclass{article}
\usepackage{amsmath,amssymb,amsthm,bbm}
\usepackage{tikz}
\usetikzlibrary{arrows, cd, calc, positioning}

\usepackage[utf8]{inputenc}
\usepackage[T1]{fontenc}
\usepackage{csquotes}
\usepackage[english]{babel}
\usepackage{hyperref}
\usepackage{stmaryrd}
\usepackage[nottoc, notlof, notlot]{tocbibind}

\newcommand{\assign}{:=}

\numberwithin{equation}{section}

\newcommand{\tmop}[1]{\ensuremath{\operatorname{#1}}}

\newcommand{\tmstrong}[1]{\textbf{#1}}

\theoremstyle{plain}
\newtheorem{theorem}{Theorem}
\newtheorem{corollary}{Corollary}
\newtheorem{definition}{Definition}
\newtheorem{lemma}{Lemma}
\newtheorem{proposition}{Proposition}

\theoremstyle{remark}
\newtheorem{remark}{Remark}
\newtheorem{example}{Example}

\newcommand{\XXint}[3]{{\setbox}0=\text{\ensuremath{#1 #2 #3 \int}}
{\vcenter{\text{\ensuremath{#2 #3}}}}{\kern}-.5{\tmwd}0}

\newcommand{\opn}[2]{\newcommand{\1}{\}} {\opn}{\Rm{Rm}} {\opn}{\Ric{Ric}}
{\opn}{\Rc{Rc}} {\opn}{\Scal{Sc}} {\opn}{\Tr{Tr}} {\opn}{\Trac{Tr}}
{\opn}detdet {\opn}{\diam{diam}} {\opn}{\dist{dist}} {\opn}{\Im}Im
{\opn}{\div}div {\opn}{\Ker{Ker}} {\opn}expexp {\opn}{\Vol{Vol}}
{\opn}{\exph{exph}} {\opn}{\Herm{Herm}} {\opn}{\End{End}} {\opn}{\Hess{Hess}}
{\opn}{\Vol{Vol}}}
\newcommand{\N}{\mathbb{N}}

\newcommand{\contract}{{\kern}-1.5pt{\vrule} width6.0pt height0.4pt depth0pt
{\vrule} width0.4pt height4.0pt depth0pt}
\newcommand{\retract}{{\kern}-1.5pt{\vrule} width0.4pt height4.0pt depth0pt
{\vrule} width6.0pt height0.4pt depth0pt}
\newcommand{\Openbox}{{\leavevmode} {\hfil}{\vrule} width{\boxrulethickness}
{\vbox} to{\Openboxwidth{{\advance}{\Openboxwidth} -2{\boxrulethickness}
{\hrule} height {\boxrulethickness} width{\Openboxwidth}{\vfil} {\hrule}
height{\boxrulethickness}}}{\vrule} width{\boxrulethickness}{\hfil} }

\newcommand\shape{\operatorname{sh}}

\newcommand\child[2][]{\operatorname{Child}_{#1}(#2)}
\newcommand\nchild[2][]{\operatorname{\ell}_{#1}(#2)}  
\newcommand\Trees[1]{\operatorname{Tree}_{#1}}         
\newcommand\Root[1]{\rho(#1)}                           
\newcommand\emptyroot{\circ}                           
\newcommand\treeforest[1]{\operatorname{Tree}(#1)}     
\newcommand\Forests{\mathcal{F}}
\newcommand\Map[2]{\operatorname{Map}(#1,#2)}
\DeclareMathOperator{\ad}{ad}
\DeclareMathOperator{\Dyck}{Dyck}
\newcommand{\Part}{\mathcal{P}}

\newcommand\dyckpath[3]{
  \fill[cyan!25]  (#1) rectangle +(#2,#2);
  \fill[fill=lime]
  (#1)
  \foreach \dir in {#3}{
    \ifnum\dir=0
    -- ++(1,0)
    \else
    -- ++(0,1)
    \fi
  } |- (#1);
  \draw[help lines] (#1) grid +(#2,#2);
  \draw[dashed] (#1) -- +(#2,#2);
  \coordinate (prev) at (#1);
  \foreach \dir in {#3}{
    \ifnum\dir=0
    \coordinate (dep) at (1,0);
    \else
    \coordinate (dep) at (0,1);
    \fi
    \draw[line width=1pt,-stealth] (prev) -- ++(dep) coordinate (prev);
   };
}

\begin{document}

\title{Multiple Lie Derivatives and Forests}\author{\\
Florent HIVERT and Nefton PALI}\date{}\maketitle

\begin{abstract}
  We obtain a complete time expansion of the pull-back operator generated by a
  real analytic flow of real analytic automorphisms acting on analytic tensor
  sections of a manifold. Our expansion is given in terms of multiple Lie
  derivatives. Motivated by this expansion, we provide a rather simple and
  explicit estimate for higher order covariant derivatives of multiple Lie
  derivatives acting on smooth endomorphism sections of the tangent bundle of
  a manifold. We assume the covariant derivative to be torsion free. The
  estimate is given in terms of Dyck polynomials. The proof uses a new result
  on the combinatorics of rooted labeled ordered forests and Dyck polynomials.
\end{abstract}

\section{Notations and motivation}

In this paper the products over ordered sets of indices are always considered
from the left to the right with respect to the order structure of the set. For
any integer $k \geqslant 1$, we set $[k] \assign \{1, \ldots, k\}$ and we
denote by $\lambda \vDash k$ any element $\lambda \equiv \left( \lambda_1,
\ldots, \lambda_{l_{\lambda}} \right) \in \N^{l_\lambda}_{> 0} $ such that $\sum_{j = 1}^{l_\lambda} \lambda_j = k$. We denote
by $\N^l(p)$ the set of $P \equiv \left( p_1, \ldots, p_l \right) \in
\N^l$ such that $\sum_{j = 1}^l p_j = p$.

We start now with a general remark. Let $\left( \varphi_t \right)_{t \in
\left( - \varepsilon, \varepsilon \right)}$ be a real analytic family of real
analytic automorphisms of a real analytic manifold $X$ with $\varphi_0 = \tmop{id}_X$ and let $\xi
\left( t \right) \assign \dot{\varphi}_t \circ \varphi^{- 1}_t = \sum_{k
\geqslant 0} \xi_k t^k$. The pull back operator $\varphi^{\ast}_t$ acting on
real analytic sections of the tensor bundles $\left( T^{\ast}_X
\right)^{\otimes p} \otimes T^{\otimes q}_X$, with $p, q \geqslant 0$, can be
expanded as
\begin{equation}
  \label{PullExpansion} \varphi^{\ast}_t =\mathbbm{I}+ \sum_{k \geqslant 1}
  t^k \sum_{\lambda \vDash k} \prod_{j = 1}^{l_{\lambda}} L_{| \lambda |^{-
  1}_j \xi_{\lambda_j - 1}}\,,
\end{equation}
where $| \lambda |_j : = \sum^j_{r = 1} \lambda_r$. Of course we can relax the
analyticity assumption to smooth on the ground variable $x \in X$, but for our
future applications we will need to stay in the real analytic category.

We equip the Sobolev space $H^r( X, \left( T^{\ast}_X \right)^{\otimes
q} \otimes T^{\otimes p}_X)$ with the Sobolev norm $\| \cdot \|_r$
obtained using the covariant derivatives and the pointwise max norm on
multilinear forms with respect to a smooth Riemannian metric $g$. This norm is
equivalent to the usual Sobolev norm defined by means of partitions of unity.
We remind that the space $H^r( X, \left( T^{\ast}_X \right)^{\otimes q}
\otimes T^{\otimes p}_X)$ is an algebra for $r \in \mathbbm{N}$
sufficiently big and for such $r$ hold the inequality $\|u v\|_r \leqslant C_r
\|u\|_r \|v\|_r$, for some constant $C_r > 0$. From now on we fix such an $r$.
Sobolev norms are quite natural for Hardy spaces. In any case the estimate in
the main theorem below hold with respect to any algebra norm.

A case of major interest is when the pull back operator act on analytic
endomorphism sections of the tangent bundle. Indeed this is the case when we
consider a complex structure $J$ over a complex manifold $X$ and we wish to
study the dynamics of the flow $\varphi^{\ast}_t J$. As explained in \cite{P-S, Pal2, Pal3}, among others, this is a central problem in complex
differential geometry related with a strong version of the Hamilton-Tian
conjecture (See \cite{Pal3}). For the applications it is very important to have an
explicit and simple estimate of the multiple Lie derivatives that appear in
the expansion (\ref{PullExpansion}) and their higher order covariant
derivatives. This is provided by the following result, which is our main
theorem.

\begin{theorem}
  {\tmstrong{$\left( {\bf Main Theorem} \right)$}} \label{MainTeo}Let $\nabla$ be the extension to tensor sections of any torsion free connection and let
  $A$ be a smooth endomorphism section of the tangent bundle. Then for any
  family of smooth vector fields $\left( \xi_j \right)^k_{j = 1}$ the estimate
  holds
  \begin{multline*}
   \frac{1}{h!}  \left\| \nabla^h \left( \prod_{j = 1}^k L_{\xi_j} \right) A
    \right\|_r\\
    \ \ \leqslant\ \
    C^k_r \sum_{\substack{
        P \in \Dyck \left( k \right)\\
        H \in \N^{k + 1}(h)}}
    C_P  \frac{1}{h_{k + 1} !}
    \| \nabla^{h_{k + 1} + D_{P, k}} A\|_r
    \prod_{j = 1}^k \frac{1}{h_j !}  \| \nabla^{h_j + p_j} \xi_j \|_r\,,
  \end{multline*}
  where
  \begin{align*}
    \Dyck(k) &\ \assign\
    \left\{ P \equiv ( p_1,\dots, p_k) \in \N^k
      \ \middle|\
      \sum_{1 \leqslant r \leqslant j} p_r \leqslant j, \forall j \in [k]
    \right\}\,,\\
    D_{P, j} &\ \assign\
    j - \sum_{1 \leqslant r \leqslant j} p_r\,,\\
    C_P &\ \assign\
    \prod_{j = 1}^k \left[
      2 \binom{D_{P, j - 1}}{p_j - 1} + \binom{D_{P, j - 1}}{p_j}
    \right] \\
    &\ =\ \prod_{\substack{
      1 \leqslant j \leqslant k\\
      p_j  \neq 0
    }}
  \left( 2 + \frac{D_{P, j}}{p_j} \right) \binom{D_{P, j - 1}}{p_j - 1}\,,
  \end{align*}
  with the convention that $\binom{m}{n}=0$ in $n\notin\{0,1,\dots,m\}$.
\end{theorem}
\bigskip

The elements $P$ of $\Dyck(k)$ are called \emph{Dyck vectors} of length
$k$. Each Dyck vector is associated to a \emph{Dyck monomial}
$X^P=X_1^{p_1}\cdots X_k^{p_k}$. We give below their full list for $k=1,2,3$
together with the associated $C_P$ coefficient
\begin{gather*}
  \Dyck(1) :\quad (0)\ 1\quad (1)\ 2 \\
  \Dyck(2) :\quad (00)\ 1\quad (01)\ 3\quad (02)\ 2\quad (10)\ 2\quad (11)\ 4 \\
  \Dyck(3) :
  \begin{array}{ccccccc}
    (000)\ 1&(001)\ 4&(002)\ 5&(003)\ 2&(010)\ 3&(011)\
    9&(012)\ 6\\
    (020)\ 2& (021)\ 4& (100)\ 2& (101)\ 6& (102)\ 4& (110)\ 4& (111)\ 8
  \end{array}
\end{gather*}
It easy to see that the cardinality $|\Dyck(k)|$ is the $k+1$-th Catalan
number $C_{k+1}$ where $C_k=\frac1{k+1}\binom{2k}{k}$. Indeed $C_k$ is
known~(see~\cite{OEIS} sequence
\href{http://oeis.org/A000108}{A000108}) to be the number of so called Dyck
path, that is lattice path on the grid $\N\times\N$, starting from $(0,0)$
ending at $(k,k)$ with only North and East step and staying under the
diagonal. As illustrated below, such a path can be bijectively encoded by the
length of the vertical segment, omitting the last one. Requiring that the path
stay below the diagonal is equivalent to the condition $\sum_{1 \leqslant r
  \leqslant j} p_r \leqslant j$
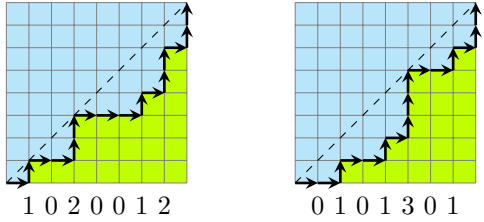
\[
\begin{tikzpicture}[scale=0.3, baseline=-40]
  \dyckpath{0,-9}{8}{0,1,0,0,1,1,0,0,0,1,0,1,1,0,1,1};
  \coordinate (prev) at (0,-10);
  \foreach \h in {1,0,2,0,0,1,2}{
    \coordinate (prev) at ($ (prev) + (1,0) $);
    \node at (prev) {$\h$};
  };
\end{tikzpicture}
\qquad\qquad
\begin{tikzpicture}[scale=0.3, baseline=-40]
  \dyckpath{0,-9}{8}{0,0,1,0,0,1,0,1,1,1,0,0,1,0,1,1};
  \coordinate (prev) at (0,-10);
  \foreach \h in {0,1,0,1,3,0,1}{
    \coordinate (prev) at ($ (prev) + (1,0) $);
    \node at (prev) {$\h$};
  };
\end{tikzpicture}
\]
\begin{example}\label{ex-comput-Cp}
We now illustrate the computation of $C_P$. Let $P=(0,1,0,1,3,0,1)$. Then the
value of $D_{P,j}$ are given by the following array:
\[\begin{array}{|c|cccccccc|}
  \hline
  j       & 0 & 1 & 2 & 3 & 4 & 5 & 6 & 7 \\
  \hline
  p_j  & & 0 & 1 & 0 & 1 & 3 & 0 & 1 \\
  D_{P,j} & 0 & 1 & 1 & 2 & 2 & 0 & 1 & 1 \\
  \hline
\end{array}\]
So that
\begin{align*}
  C_p &=
  \left[2\binom{0}{-1}+\binom{0}{0}\right]
  \left[2\binom{1}{0}+\binom{1}{1}\right]
  \left[2\binom{1}{-1}+\binom{1}{0}\right]\times\\
  &\times\left[2\binom{2}{0}+\binom{2}{1}\right]
  \left[2\binom{2}{2}+\binom{2}{3}\right]
  \left[2\binom{0}{-1}+\binom{0}{0}\right]
  \left[2\binom{1}{0}+\binom{1}{1}\right]\\
  &=72\,.
\end{align*}
\end{example}

\section{Proof of the expansion formula (\ref{PullExpansion})}

We remind first (see for instance lemma 27 in the sub-section 19.4 on page 916 of \cite{Pal1}) the well known derivation rule
\begin{equation}
  \label{generLIE}  \frac{d}{d t}  \left( \varphi^{\ast}_t \alpha_t \right)
  \ =\
  \varphi^{\ast}_t \left( \frac{d}{d t} \alpha_t + L_{\xi_t} \alpha_t \right),
\end{equation}
for any curve $t \longmapsto \alpha_t \in \left( T^{\ast}_X \right)^{\otimes
p} \otimes T^{\otimes q}_X$. What we need to prove in order to obtain
(\ref{PullExpansion}) is the formula
\begin{equation*}
  \frac{1}{k!}  \frac{d^k}{d t^k} _{\mid_{t = 0}}  \left( \varphi^{\ast}_t
  \alpha \right)\ =\ \sum_{\lambda \vDash k}  \left[ \prod_{j =
  1}^{l_{\lambda}} L_{| \lambda |^{- 1}_j \xi_{\lambda_j - 1}} \right]\alpha,
\end{equation*}
with $\alpha \in \left( T^{\ast}_X \right)^{\otimes p} \otimes T^{\otimes
q}_X$, that we rewrite under the form
\begin{align*}
  \frac{d^k}{d t^k} _{\mid_{t = 0}}  \left( \varphi^{\ast}_t \alpha \right)
  &\ =\
  \sum_{\lambda \vDash k} C_{\lambda}
  \left[
    \prod_{j = 1}^{l_{\lambda}} L_{\xi_0^{\left( \lambda_j - 1 \right)}}
  \right] \alpha\,,\\
  C_{\lambda} &
  \ =\  \frac{| \lambda | !}{ \prod_{j = 1}^r
    \left[
      \left(\lambda_j - 1 \right)!\ |\lambda|_j
    \right] }\,,\\
  \xi_t^{\left( k \right)} & \assign \frac{d^k \xi_t}{d t^k}\,,
\end{align*}
with the convention $0! = 1$. We will prove the more general formula
\begin{equation*}
  \frac{d^k}{d t^k} \varphi^\ast_t \ = \
  \sum_{\lambda \vDash k} C_\lambda
  \varphi_t^\ast  \prod_{j = 1}^{l_\lambda}
  L_{\xi_t^{\left( \lambda_j - 1\right)}}\,,
\end{equation*}
by induction. (Obviously the above formula is true for $k = 1$.) Taking one
more derivative we obtain thanks to (\ref{generLIE})
\begin{equation*}
  \frac{d^{k + 1}}{d t^{k + 1}} \varphi^{\ast}_t
  \ = \
  \sum_{\lambda \vDash k} C_{\lambda} \varphi_t^{\ast}
  \left[
    \sum_{s = 1}^{l_{\lambda}} \prod_{j = 1}^{l_{\lambda}}
    L_{\xi_t^{\left( \lambda_j + \delta_{j, s} - 1 \right)}}
    + L_{\xi_t} \prod_{j = 1}^{l_{\lambda}} L_{\xi_t^{\left(\lambda_j - 1
        \right)}}
  \right] .
\end{equation*}
If we identify formally the product $\varphi_t^*\prod_{j = 1}^{l_{\lambda}}
L_{\xi_t^{\left( \lambda_j - 1 \right)}}$ with the composition $\lambda$ then
the previous sum corresponds to the formal sum of compositions
\begin{equation}
\label{equ-b-compo}
S_k\ \assign\ \sum_{\lambda \vDash k} C_{\lambda}  \left[ \lambda' + \left( 1, \lambda
  \right)_{_{_{_{}}}} \right],
\end{equation}
where
\begin{alignat}{1}
  \lambda' &\ \assign \ \sum_{j = 1}^{l_{\lambda}} \left( \lambda_1, \ldots,
  \lambda_j + 1, \ldots, \lambda_{l_{\lambda}} \right),\\
  \left( 1, \lambda \right) &\ \assign \ \left( 1, \lambda_1, \ldots,
  \lambda_{l_{\lambda}} \right) .
\end{alignat}
We observe that the operation which associates to any composition $\lambda$
the components of the formal sum $\lambda' + \left( 1, \lambda \right)$,
generates all the compositions of $k + 1$. For any $\lambda\vDash k$ and
$\Lambda\vDash k+1$ let us write $\lambda\to\Lambda$ if $\Lambda=\left(
  \lambda_1, \ldots, \lambda_j + 1, \ldots, \lambda_{l_{\lambda}} \right)$ for
some $j$ or $\Lambda=(1,\lambda)$. Then Equation~\ref{equ-b-compo} rewrites as
\begin{equation}
  S_k = \sum_{\lambda \vDash k} C_{\lambda}
    \sum_{
      \substack{\Lambda \vDash k+1 \\ \lambda\to\Lambda}
} \Lambda\,.
\end{equation}
Exchanging the two sums yield
\begin{equation}
  S_k = \sum_{\Lambda \vDash k+1} \left(
    \sum_{
      \substack{\lambda \vDash k \\ \lambda\to\Lambda}} C_{\lambda}
    \right)\Lambda\,.
\end{equation}
We observe that $\Lambda\vDash k+1$ being fixed, the $\lambda \vDash k$ such
that $\lambda\to\Lambda$ are obtained either by removing $1$ in front of
$\Lambda$ (if $\Lambda$ starts with $1$) or by decreasing any component of
$\Lambda$ greater that $2$.

The conclusion of the induction will follow from the equality $C_{\Lambda} =
C'_{\Lambda}$, which rewrites in a more explicit form as
\begin{equation*}
  \frac{|\Lambda|!}{\prod_{j = 1}^{l_{\Lambda}}
    \left[ (\Lambda_j - 1)!\ |\Lambda|_j \right]}
  \ = \
  \sum_{s=1}^{l_{\Lambda}}
  \frac{(|\Lambda | - 1)!}%
       {\prod_{j=1}^{l_{\Lambda}} \left[(\Lambda_j - \delta_{s, j} - 1)_{+}!\
           \sum^j_{t = 1} (\Lambda_t - \delta_{s, t}) \right]},
\end{equation*}
where $\left( a \right)_+ \assign \max \{a, 0\}$. Symplifing the common factor
$(|\Lambda| - 1)!$, rearranging and setting $l \assign
l_{\Lambda}$ we infer that the previous equality is equivalent to
\begin{equation*}
  \sum_{s = 1}^l \Lambda_s
  \ = \ \sum_{s = 1}^l\ \prod_{j = 1}^l
  \frac{(\Lambda_j - 1)!\ \sum^j_{t = 1} \Lambda_t}%
       {\left[ (\Lambda_j - \delta_{s, j} - 1 )_{+}!\
           \sum^j_{t = 1} (\Lambda_t - \delta_{s, t}) \right] }\,.
\end{equation*}
The later rewrites in a simpler way as
\begin{equation}
  \label{keySumInduct}
  \sum_{s = 1}^l \Lambda_s = \sum_{s = 1}^l
  (\Lambda_s - 1)
  \prod_{j = s}^l \frac{\sum_{t = 1}^j \Lambda_t}{ \sum_{t = 1}^j \Lambda_t - 1}\,.
\end{equation}
We show (\ref{keySumInduct}) by induction on $l$. The equality
(\ref{keySumInduct}) is obvious for $l = 1$. We decompose the sum
\begin{align*}
  & \sum_{s = 1}^{l + 1} \left( \Lambda_s - 1 \right) \prod_{j = s}^{l + 1}
  \frac{ \sum^j_{t = 1} \Lambda_t}{ \sum^j_{t = 1} \Lambda_t - 1}
  \\
  &\ =\ \left(
  \Lambda_{l + 1} - 1 \right)  \frac{ \sum^{l + 1}_{t = 1} \Lambda_t}{ \sum^{l
  + 1}_{t = 1} \Lambda_t - 1}\\
  &\ +\ \frac{ \sum^{l + 1}_{t = 1} \Lambda_t}{ \sum^{l + 1}_{t = 1}
  \Lambda_t - 1}  \sum_{s = 1}^l \left( \Lambda_s - 1 \right) \prod_{j = s}^l
  \frac{ \sum^j_{t = 1} \Lambda_t}{ \sum^j_{t = 1} \Lambda_t - 1} \\
  &\ =\ \left( \Lambda_{l + 1} - 1 \right)  \frac{ \sum^{l + 1}_{t = 1}
  \Lambda_t}{ \sum^{l + 1}_{t = 1} \Lambda_t - 1} + \frac{ \sum^{l +
  1}_{t = 1} \Lambda_t}{ \sum^{l + 1}_{t = 1} \Lambda_t - 1}  \sum_{s = 1}^l
  \Lambda_s\,,
\end{align*}
by the inductive assumption. We conclude
\begin{align*}
  \sum_{s = 1}^{l + 1} (\Lambda_s - 1)
  \prod_{j = s}^{l + 1}
     \frac{\sum_{t = 1}^j\Lambda_t}{\sum_{t = 1}^j \Lambda_t - 1}
  &\ =\ \left(
    \sum^{l + 1}_{t = 1} \Lambda_t - 1
  \right)
  \frac{ \sum_{t = 1}^{l + 1}\Lambda_t}{\sum^{l + 1}_{t = 1} \Lambda_t - 1} \\
  &\ =\ \sum_{t = 1}^{l + 1} \Lambda_t\,,
\end{align*}
which is the equality (\ref{keySumInduct}) for $l + 1$.

\paragraph{Combinatorial proof.}
We give now a second more combinatorial proof of the formula giving the
$C_\lambda$. Recall that we encoded the product
\begin{equation}
\varphi_t^*\prod_{j = 1}^{l_{\lambda}} L_{\xi_t^{\left( \lambda_j - 1 \right)}}\,,
\end{equation}
with the composition $\lambda$. We denote $D$ the linear operator acting on
formal linear combination of compositions defined by
\begin{equation}
  D(\lambda) := \lambda' + (1,\lambda) = \sum_{\lambda\to\Lambda} \Lambda\,,
\end{equation}
where we defined the relation $\to$ by $\lambda\to\Lambda$ if
$\Lambda=\left( \lambda_1, \ldots, \lambda_j + 1, \ldots,
  \lambda_{l_{\lambda}} \right)$ for some $j$ or $\Lambda=(1,\lambda)$.  Then
$$
\frac{d}{dt}\left(\varphi_t^*\prod_{j = 1}^{l_{\lambda}}
L_{\xi_t^{\left( \lambda_j - 1 \right)}}\right)
$$
is encoded by $D(\lambda)$. So that to compute $\frac{d^{k}}{d t^{k}}
\varphi^{\ast}_t$ we need to compute the coefficients $c_k$ of the expansion
of $D^k(())=\sum_{\lambda\vDash k} c_\lambda\lambda$ where $()$ denotes the
empty composition of length $0$. The relation $\to$ is
illustrated in Figure~\ref{graphe}, together with the coefficient $c_\lambda$.

By definition of $D$ the coefficient $c_\Lambda$ of
each node $\Lambda$ is the sum of the coefficients of its antecedents by the
relation $\lambda\to\Lambda$. As a consequence it is equal to the number of
pathes from $()$ to $\Lambda$, that is sequences
\begin{equation*}
()=\lambda^0\to\lambda^1\to\lambda^2\to\dots\to\lambda^k=\Lambda
\end{equation*}
such that the relation $\lambda^i\to\lambda^{i+1}$ holds for any $i$ such that
$0\leq i< k$.
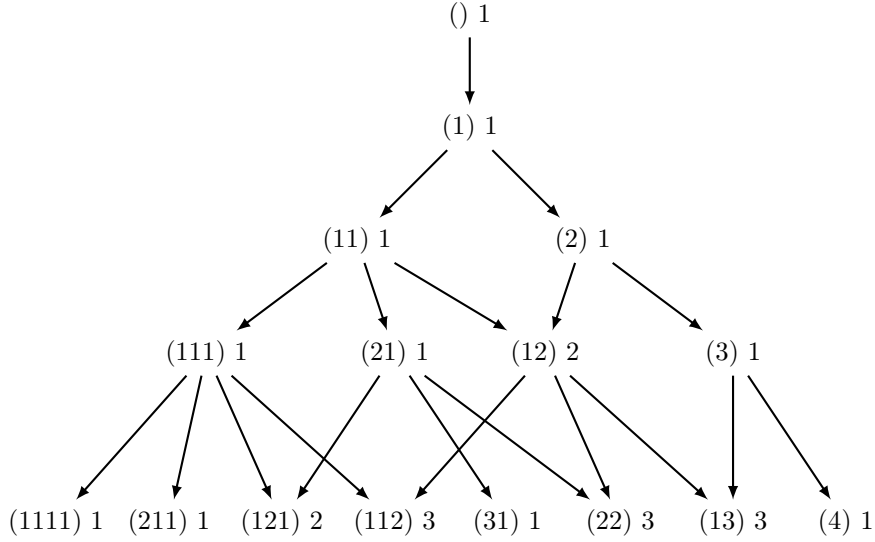
\begin{figure}[ht]
\[
\begin{tikzpicture}[yscale=1.5]
  \node (0) at (5.5, 5) {$()\ 1$};
  \node (1) at (5.5, 4) {$(1)\ 1$};
  \node (11) at (4, 3) {$(11)\ 1$};
  \node (2) at (7, 3) {$(2)\ 1$};
  \node (111) at (2, 2) {$(111)\ 1$};
  \node (21) at (4.5, 2) {$(21)\ 1$};
  \node (12) at (6.5, 2) {$(12)\ 2$};
  \node (3) at (9, 2) {$(3)$\ 1};
  \node (1111) at (0, 0.5) {$(1111)$\ 1};
  \node (211) at (1.5, 0.5) {$(211)$\ 1};
  \node (121) at (3, 0.5) {$(121)$\ 2};
  \node (112) at (4.5, 0.5) {$(112)$\ 3};
  \node (31) at (6, 0.5) {$(31)$\ 1};
  \node (22) at (7.5, 0.5) {$(22)$\ 3};
  \node (13) at (9, 0.5) {$(13)$\ 3};
  \node (4) at (10.5, 0.5) {$(4)$\ 1};
  \tikzstyle{fils}=[->,thick,>=latex]
  \draw[fils] (0)--(1);
  \draw[fils] (1)--(11);
  \draw[fils] (1)--(2);
  \draw[fils] (11)--(111);
  \draw[fils] (11)--(21);
  \draw[fils] (11)--(12);
  \draw[fils] (2)--(12);
  \draw[fils] (2)--(3);
  \draw[fils] (111)--(1111);
  \draw[fils] (111)--(211);
  \draw[fils] (111)--(121);
  \draw[fils] (111)--(112);
  \draw[fils] (21)--(121);
  \draw[fils] (21)--(31);
  \draw[fils] (21)--(22);
  \draw[fils] (12)--(112);
  \draw[fils] (12)--(22);
  \draw[fils] (12)--(13);
  \draw[fils] (3)--(13);
  \draw[fils] (3)--(4);
\end{tikzpicture}
\]
\caption{\label{graphe}The
  relation $\lambda\to\Lambda$ and the coefficient $c_\lambda$}
\end{figure}

The striking observation is that the number of such paths starting from
$()$ to any composition $\Lambda$ of sum $k$ is equal to the number of
partitions of the set $[k]=\{1,\dots,k\}$. Recall that a partition of a
set $S$ is a set $\Pi = \{\Pi_1, \dots, \Pi_r\}$ of non empty disjoints sets
$\Pi_i$ whose union is $S$. The number of partition of $[k]$ is know as the
$k$-th Bell numbers. The first values are
$$
1, 1, 2, 5, 15, 52, 203, 877, 4140, 21147,
$$
and one can check on Figure~\ref{graphe} that the sum of the coefficient of
the composition of sum $4$ is indeed $15$. So we can expect that there is a
bijection between the set of paths from $()$ to $\Lambda$ and the set of
partition $\Pi$ of $k$ verifying certain constraints depending on $\Lambda$.

We now describe such a constraint. First of all, we need a unique compact way
to write a partition $\Pi=\{\Pi_1, \dots, \Pi_r\}$ of $k$. So we write the
element of each $\Pi_i$ in the natural order, separating the $\Pi_i$ by
vertical bars $|$ and sorting the $\Pi_i$ among themselves according to their
\emph{largest} element. We call this ordering $(\Pi_1, \dots, \Pi_r)$ the
\emph{normal ordering}. For example $\{\{6\}, \{1\}, \{7, 2, 4\}, \{5,
3\}\}$ which is a partition of $[7]$ is rather written in the order $\{\{1\},
\{3, 5\} , \{6\}, \{2, 4, 7\}\}$ in the compact way $1|35|6|247$.
\begin{definition}
  Given a partition $\Pi$ of $k$ with normal ordering $(\Pi_1,\dots,\Pi_r)$,
  we call the \emph{shape} of $\Pi$ and denote $\shape(\Pi)$ the composition
  $(|\Pi_1|, \dots, |\Pi_r|)$.
\end{definition}
For example $\shape(1|35|6|247)=(1,2,1,3)$. Then we claim that
\begin{proposition}\label{prop-coeff-part}
  For any composition $\Lambda\vDash k$, the coefficient $c_\Lambda$ is the
  number of partitions $\Pi$ such that $\shape(\Pi)=\Lambda$.
\end{proposition}
Before going to the proof we need some extra combinatorial ingredients. Given
a partition $\Pi$ of $k>0$ we denote $\Pi^-$ the partition of $k-1$ obtained
by decreasing by $1$ all the numbers in the elements of $\Pi$, removing the
obtained $0$ and its set as well if it is a singleton. We moreover define
$\Pi^{(n)}$ by $\Pi^{(0)}=\Pi$ and $\Pi^{(n)}=(\Pi^{(n-1)})^-$. For example
$1|35|6|247^-=24|5|136$ and $345|26|17^-=234|15|6$. Here is the sequence
$1|35|6|247^{(n)}$ for $n=0,\dots 7$, where the second row gives the shapes:
\[\begin{array}{cccccccc}
1|35|6|247 & 24|5|136 & 13|4|25 & 2|3|14 & 1|2|3 & 1|2 & 1 & \emptyset \\
(1,2,1,3) & (2,1,3)   & (2,1,2) & (1,1,2) & (1,1,1) & (1,1) & (1) & ()
\end{array}\]
We notice the following obvious lemma which we just illustrated.
\begin{lemma}\label{lemma-link}
  For all $\Pi$ partition of $k>0$, $\shape(\Pi^-)\to\shape(\Pi)$.
\end{lemma}
We now turn to the proof of Proposition~\ref{prop-coeff-part}.

\begin{proof}
  Let's denote $Q_\Lambda:= \{ \Pi \mid \shape(\Pi)=\Lambda\}$ and $P_\Lambda$
  the set of paths
  \begin{equation*}
    ()=\lambda^0\to\lambda^1\to\lambda^2\to\dots\to\lambda^k=\Lambda
  \end{equation*}
  from $()$ to $\Lambda$. To prove than
  $|Q_\Lambda|=|P_\Lambda|$ we define a bijection $F:Q_\Lambda\to P_\Lambda$
  by
  \begin{equation}
    F(\Pi) = (\shape(\Pi^{(k)}), \shape(\Pi^{(k-1)}),
    \dots, \shape(\Pi^{(1)}), \shape(\Pi^{(0)}))\,.
  \end{equation}
  Observe that $F(\Pi)$ is obtained by appending $\shape(\Pi)$ to $F(\Pi^-)$.
  Thanks to Lemma~\ref{lemma-link}, this prove that $F(\Pi)\in P_\Lambda$. We
  now need to show that $F$ is a bijection, that is, given a path
  \begin{equation*}
    p = (()=\lambda^0\to\lambda^1\to\lambda^2\to\dots\to\lambda^k=\Lambda)\,,
  \end{equation*}
  we need to show that there is a unique partition $\Theta$ such that
  $F(\Theta) = p$. We proceed by induction on $k$. First, for $k=0$, we
  observe that there is only one partition of shape $()$, namely the empty
  partition $\emptyset$. Now suppose that $\Pi=(\Pi_1,\dots,\Pi_r)$ is the
  unique partition such that $F(\Pi)=(\lambda^0,\dots,\lambda^k)$. We only
  need to show that there is a unique partition $\Theta$ such that
  $\Theta^-=\Pi$ and $\shape(\Theta)=\lambda^{k+1}$. Recall that if
  $\lambda^k\to\lambda^{k+1}$, they are two possibilities:
  \begin{itemize}
  \item Either $\lambda^{k+1}= (1, \lambda^k)$, in this case, the only
    possible $\Theta$ is
    \[
    \Theta = (\{1\}, \Pi_1+1,\dots,\Pi_r+1)\,,
    \]
    where for any set $S$ of integers, $S+1:=\{i+1\mid i\in S\}$.
  \item Or writing $\lambda^{k}= (\lambda^{k}_1,\dots,\lambda^{k}_{l})$ there
    exists $j\leq l$ such that
    \[
    \lambda^{k+1}=
    (\lambda^{k}_1,\dots,\lambda^{k}_j+1,\dots,\lambda^{k}_{l})\,.
    \]
    Since $\lambda^{k}=\shape(\Pi)$, it makes sense to define
    \[
    \Theta = (\Pi_1+1,\dots,\{1\}\cup(\Pi_j+1), \dots\Pi_r+1)\,.
    \]
    and again it is the only possibility.
  \end{itemize}
  This conclude the proof by induction on $k$.
\end{proof}

To finish the combinatorial proof of Formula~(\ref{PullExpansion}), we still need
to prove the following:
\begin{proposition}
  For any composition $\lambda\vDash k$, the number of partition $\Pi$ of
  $[k]$ of shape $\lambda$ is given by
  \begin{equation}
    c_{\lambda}\ =\ \frac{| \lambda | !}{ \prod_{j = 1}^r \left[ \left(
          \lambda_j - 1 \right)_{_{_{_{}}}}!\ |\lambda|_j \right] }\,.
\end{equation}
\end{proposition}

\begin{proof}
  We proceed by induction on the length $r$ of any composition $\lambda$ of
  any sum $k$. If $r=0$, then
  $\lambda=()$, and the denominator product is empty so that $c_\lambda=1$
  which is correct since the only partition is the empty one. Now to choose a
  partition $\Pi=\{\Pi_1, \dots, \Pi_{r+1}\}$ of shape
  $\lambda=(\lambda_1,\dots, \lambda_{r+1})$ of $k$, we first need to choose
  the elements which belongs to $\Pi_{r+1}$ (which must contains at least $k$
  due to the normal ordering). To get the correct shape, there must be
  $\lambda_{r+1}-1$ elements different from $k$ in $\Pi_{r+1}$ that must be
  chosen in $[k-1]=[|\lambda|-1]$. So the number of such choices is
  \begin{align*}
    \binom{|\lambda|-1}{\lambda_{r+1}-1}
    &\ =\
    \frac{(|\lambda|-1)!}{(\lambda_{r+1}-1)!(\lambda_1+\dots+\lambda_{r})!} \\
    &\ =\  \frac{|\lambda|!}{(\lambda_{r+1}-1)!|\lambda|_{r+1}|\lambda|_r!}\,,
  \end{align*}
  since $|\lambda|=|\lambda|_{r+1}$.  We need now to choose a partition
  $\Theta$ of shape $\mu:=(\lambda_1,\dots,\lambda_{r})$ of the remaining
  numbers.  By naturally renumbering them, there are as many choices for
  $\Theta$ as partitions of $[|\lambda|_r]$ of shape
  $(\lambda_1,\dots,\lambda_{r})=\mu$. By induction they are $c_\mu$ of
  them. We therefore obtain
  \begin{equation}
    \frac{|\lambda|!}{(\lambda_{r+1}-1)!|\lambda|_{r+1}|\lambda|_r!}\,
    \frac{| \lambda |_r !}{ \prod_{j = 1}^r \left[ \left(
          \lambda_j - 1 \right)_{_{_{_{}}}}!\ |\lambda|_j \right] }\,,
  \end{equation}
  which simplifies to the announced result.
\end{proof}
\begin{remark}
  The shape map $\shape$ and the coefficients $c_\lambda$ have a nice
  algebraic interpretation in terms of combinatorial Hopf
  algebras~\cite{HNT-Comm}. Indeed set partition index the monomial basis
  $(\mathbf{M}_\Pi)$ of the algebra \textbf{WSym} of symmetric function in
  non-commutative variables. Then the operation $\Pi\mapsto\Pi^-$, is encoded
  in the non-commutative product of \textbf{WSym} as
  \begin{equation}
    \mathbf{M}_{\{\{1\}\}}\mathbf{M}_\Pi
    = \sum_{\Theta\ |\ \Pi=\Theta^-}    \mathbf{M}_{\Theta}\,.
  \end{equation}
  Then~\cite[Section 3.7]{HNT-Comm} consider a quotient of \textbf{WSym} by the
  so-called stalactic congruence. To match our setting we need the right sided
  stalactic congruence defined by
  \begin{equation}
    a\ w\ a \equiv w\ a\ a
  \end{equation}
  for all $a\in A$ and $w\in A^*$. This quotient amount to identify
  $\mathbf{M}_\Pi$ and $\mathbf{M}_\Theta$ if and only if
  $\shape(\Pi)=\shape(\Theta)$ leading naturally to a base $(N_\lambda)$ of
  the quotient. As a consequence our $c_\lambda$ are nothing but the
  coefficients of the expansion
  \begin{equation}
    \frac{1}{1-N_{(1)}} =\sum_\lambda c_\lambda N_\lambda\,.
  \end{equation}
\end{remark}
\section{Multiple covariant derivatives and trees}

Let $X$ be a smooth manifold and let $\nabla$ be a covariant derivative
operator acting on the smooth sections of the tangent bundle $T_X$. We will
still denote by $\nabla$ its natural extension over tensors. For any subset $S
\subset \N_{>0}$ we consider a family of vector fields
$\left( \xi_p \right)_{p \in S}$ and a smooth section $A$ of the tensor bundle
$\left( T^{\ast}_X \right)^{\otimes q} \otimes T^{\otimes p}_X$.

It is known since Cayley~\cite{Cayley1857} that trees are the right tool to
manipulate nested iterated derivative (see also~\cite{Manchon,HNT-DMTCS}). He
actually invented the very notion of tree for that exact purpose. In this
paper we will have to deal with expression such as
\begin{equation*}
  \nabla^3_{\nabla^1_{\xi_2}{\xi_3}
    \otimes \xi_5
    \otimes
    \nabla^2_{\xi_1\otimes\xi_4}{\xi_6}}A
  \ \equiv\
  \left(
    \left(\xi_2\neg\nabla^1{\xi_3}\right)
    \otimes \xi_5
    \otimes
    \left((\xi_1\otimes\xi_4)\neg\nabla^2{\xi_6}\right)
  \right)\neg\nabla^3A\,.
\end{equation*}
We will manipulate them using trees. For example the previous expression is
much easier to read if written as in the left of
Figure~\ref{fig-def-tree}. Moreover, since there is a lot of redundant
information such as the $\nabla^i$ and the $\xi$, we will reduce it to the
right of Figure~\ref{fig-def-tree}. We remark that contrary to nature we
picture trees growing from top to bottom.
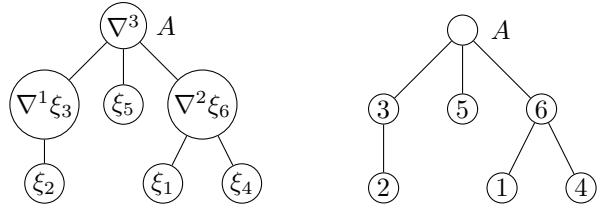
\begin{figure}[h]
  \centering
  \[
\begin{tikzpicture}[scale=0.7, inner sep=0.3mm]
  \node[right=0.4cm] {$A$};
  \node[circle, draw] (RT) {$\nabla^3$}
  child {node[circle, draw] {$\nabla^1{\xi_3}$}
    child {node[circle, draw] {$\xi_2$}}
  }
  child {node[circle, draw] {$\xi_5$}}
  child {node[circle, draw] {$\nabla^2{\xi_6}$}
    child {node[circle, draw] {$\xi_1$}}
    child {node[circle, draw] {$\xi_4$}}
  };
\end{tikzpicture}
\qquad\qquad
\begin{tikzpicture}[scale=0.7, inner sep=0.3mm, minimum size=0.4cm]
  \node[right=0.3cm] {$A$};
  \node[circle, draw] {}
  child {node[circle, draw] {$3$}
    child {node[circle, draw] {$2$}}
  }
  child {node[circle, draw] {$5$}}
  child {node[circle, draw] {$6$}
    child {node[circle, draw] {$1$}}
    child {node[circle, draw] {$4$}}
  };
\end{tikzpicture}
\]
\caption{A nested derivative and its corresponding tree}
\label{fig-def-tree}
\end{figure}

We now define formally the kinds of tree we need in this paper. We remark that
as depicted in Figure~\ref{fig-def-tree}, they are no repeated labels in our
trees so that we don't have to distinguish between a node and its label. If we
orient the edge bottom-up, such a tree is just a graph of a partial function
which is loopless. Moreover, for reasons that will become apparent latter, our
trees have some order requirement which ensure the loopless property.
\begin{definition}
  Let $S$ be a finite totally ordered set. A (rooted, labeled)
  \emph{strictly decreasing forest} $F$ on $S$ is a partially defined function
  $F: S\to S$ such that for any $\nu$ where $F$ is defined $F(\nu)>\nu$ holds.

  Elements of $S$ are called \emph{nodes} of $F$. Nodes where $F$ is not
  defined are called \emph{roots} of $F$. A forest $T$ with only one root is
  called a \emph{tree}, in this case we denote the root $\Root{T}$.

  The node $F(\nu)$ is called the \emph{father} of $\nu$.  The preimages of
  $\mu\in F^{-1}(\nu)$ by $F$ are called the \emph{children} of $\nu$ and we
  denote their set $\child[F]{\nu}$ or even $\child{\nu}$ if $F$ is clear from
  the context and their number by $\nchild[F]{\nu}$ or
  $\nchild{\nu}$. Finally, When depicting a tree, we always draw the children
  in \emph{increasing order from left to right}.

\end{definition}
Note that the condition ensure that a non-empty forest must have at least one
root, namely $\max S$.

We now fix some more terminology and notations: we often write the forest $F$
for its underlying set $S$ such as in $\nu\in F$ meaning $\nu\in S$. For a
tree $T$, we also write $\nu\in T^*$ for $\nu\in T\setminus\{\Root{T}\}$.

We say that $\nu$ is a \emph{brother} of $\mu$, if they have the same
father. Also $\nu$ is a \emph{descendant} (resp. a \emph{strict descendant})
of $\mu$ if there is a $i\geq0$ (resp. $i>0$) such that $\mu=F^i(\nu)$. Given
$\mu\in S$, the set of descendant of $\mu$ determines a natural tree called the
\emph{subtree} of $F$ rooted at $\mu$ and denoted $F_\mu$. A tree $T$ of a
forest $F$ is a subtree whose root $\Root{T}$ is also a root of the forest
$F$. We denote their set with $\treeforest{F}$. Of course there are as many
trees in a forest as roots, their number is denoted by $l_F$. We denote
$F^\dagger$ the forest obtained from $F$ by removing the tree rooted at its
maximal element.

In this paper, most of the forest and trees will have their set of nodes
contained in $\N_{>0}\cup\{\emptyroot\}$. When $\emptyroot$ is present
it is the largest element and thus a root. When a root is $\emptyroot$, we
don't draw it and simply draw and empty node in the picture. For
$S\subset\N_{>0}$, we denote by $T \in \Trees{S}$, the set of trees with set
of nodes
$S^\emptyroot\assign S\cup\{\emptyroot\}$.
\bigskip

\begin{example} We show below a forest $F$. The tree $F_3$ is a subtree of
  $F$, and $F_8$ is both a subtree and a tree of $F$, that is
  $F_8\in\treeforest{F}$. We also examplify $F^\dagger$.
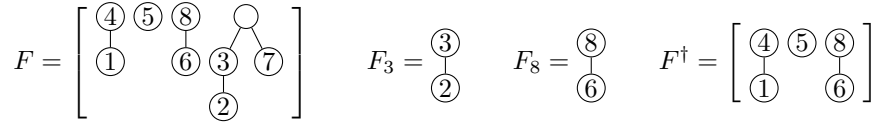
\[
  F = \left[\
 \begin{tikzpicture}[baseline=-20, scale=0.4, inner sep=0.3mm]
    \node[circle, draw] {$4$}
    child {node[circle, draw] {$1$}};
  \end{tikzpicture}\
  \begin{tikzpicture}[baseline=-20, scale=0.4, inner sep=0.3mm]
    \node[circle, draw] {$5$};
  \end{tikzpicture}\
  \begin{tikzpicture}[baseline=-20, scale=0.4, inner sep=0.3mm]
    \node[circle, draw] {$8$}
    child {node[circle, draw] {$6$}};
  \end{tikzpicture}\
  \begin{tikzpicture}[baseline=-20, scale=0.4, inner sep=0.3mm]
    \node[circle, draw] {\phantom{$\circ$}}
    child {node[circle, draw] {$3$}
      child {node[circle, draw] {$2$}}
    }
    child {node[circle, draw] {$7$}};
  \end{tikzpicture}
  \ \right]
  \qquad
  F_3 =
  \begin{tikzpicture}[baseline=-10, scale=0.4, inner sep=0.3mm]
    \node[circle, draw] {$3$}
    child {node[circle, draw] {$2$}};
  \end{tikzpicture}
  \qquad
  F_8 =
  \begin{tikzpicture}[baseline=-10, scale=0.4, inner sep=0.3mm]
    \node[circle, draw] {$8$}
    child {node[circle, draw] {$6$}};
  \end{tikzpicture}
  \qquad
  F^\dagger = \left[\
 \begin{tikzpicture}[baseline=-10, scale=0.4, inner sep=0.3mm]
    \node[circle, draw] {$4$}
    child {node[circle, draw] {$1$}};
  \end{tikzpicture}\
  \begin{tikzpicture}[baseline=-10, scale=0.4, inner sep=0.3mm]
    \node[circle, draw] {$5$};
  \end{tikzpicture}\
  \begin{tikzpicture}[baseline=-10, scale=0.4, inner sep=0.3mm]
    \node[circle, draw] {$8$}
    child {node[circle, draw] {$6$}};
  \end{tikzpicture}
  \ \right]
\]
\end{example}
It is well known that the number of forest with set of node any given set of
cardinality $k$ is $k!$ (see. \cite{OEIS} sequence
\href{http://oeis.org/A000142}{A000142}). A bijection is given
in~\cite[Section 3.2]{HNT-DMTCS}.

\begin{definition}\label{def-tree}
  For any subset $S \subset \N$ we consider a family of vector fields
  $\left( \xi_j \right)_{j \in S}$ and a smooth section $A$ of the tensor
  bundle $\left( T^{\ast}_X \right)^{\otimes q} \otimes T^{\otimes h}_X$, with
  $q, h \geqslant 0$. For any tree $T$ with set of nodes contained
  in $S^\emptyroot$ we define the nested derivative $\nabla_{\xi_{\bullet}}^T
  \xi_{\Root{T}}$ of $\xi_{\Root{T}}$, with $\xi_{\emptyroot} \assign A$ by the
  inductive formula
  \begin{equation*}
    \nabla_{\xi_{\bullet}}^T \xi_{\Root{T}}
    \ \assign \
    \left( \bigotimes_{\nu \in \child{\Root{T}}}
      \nabla^{T_{\nu}}_{\xi_{\bullet}}
      \xi_{\Root{T_{\nu}}} \right)
    \neg \nabla^{l_{\Root{T}}} \xi_{\Root{T}}\,.
  \end{equation*}
  In particular if $S^\circ$ is the set of nodes of $T$, then
  \begin{align*}
    \nabla_{\xi_{\bullet}}^T A
    &\ =\
    \left( \bigotimes_{\nu \in \child{\Root{T}}}
      \nabla^{T_{\nu}}_{\xi_{\bullet}} \xi_{\nu}
    \right)
    \neg \nabla^{l_{\Root{T}}} A\,,\\
    \nabla^{T_{\nu}}_{\xi_{\bullet}} \xi_{\nu}
    &\ =\
    \left( \bigotimes_{n \in \child{\nu}}
      \nabla^{T_n}_{\xi_{\bullet}} \xi_n \right)
    \neg \nabla^{l_{\nu}} \xi_{\nu}\,.
  \end{align*}
\end{definition}
See Figure~\ref{fig-def-tree} for an example.

If we now apply recursively the chain rule
\begin{equation}
  \nabla_{\xi} \nabla^k_{\Xi_1\otimes\dots\otimes\Xi_k}A =
  \nabla^{k+1}_{\xi\otimes\Xi_1\otimes\dots\otimes\Xi_k}A +
  \sum_{j=1}^k\nabla^k_{\Xi_1\otimes\dots\otimes\nabla_\xi\Xi_j\otimes\dots\otimes\Xi_k}A
\end{equation}
to a tree, it writes as
\begin{equation}
  \nabla_{\xi_j}\nabla_{\xi_{\bullet}}^T A =
  \sum_{U}\nabla_{\xi_{\bullet}}^U A\,,
\end{equation}
where the sum goes along the set of trees $U$ obtained by grafting $j$ to the
left of any nodes of $T$. As a consequence, there are as many terms in this sum
as nodes of $T$. We give here an example where a tree $T$ stands for
$\nabla^T_{\xi_{\bullet}}$:
\begin{equation*}
\nabla_{\xi_1}\left(
\begin{tikzpicture}[baseline=-20, scale=0.4, inner sep=0.3mm]
  \node[circle, draw] {\phantom{$A$}}
  child {node[circle, draw] {$3$}}
  child {node[circle, draw] {$6$}
    child {node[circle, draw] {$2$}}
  }
  ;
\end{tikzpicture}\,A\right)
\ =\
\begin{tikzpicture}[baseline=-20, scale=0.4, inner sep=0.3mm]
  \node[circle, draw] {\phantom{$A$}}
  child {node[circle, draw] {$1$}}
  child {node[circle, draw] {$3$}}
  child {node[circle, draw] {$6$}
    child {node[circle, draw] {$2$}}
  }
  ;
\end{tikzpicture}\,A
\ +\
\begin{tikzpicture}[baseline=-20, scale=0.4, inner sep=0.3mm]
  \node[circle, draw] {\phantom{$A$}}
  child {node[circle, draw] {$3$}
    child {node[circle, draw] {$1$}}
  }
  child {node[circle, draw] {$6$}
    child {node[circle, draw] {$2$}}
  }
  ;
\end{tikzpicture}\,A
\ +\
\begin{tikzpicture}[baseline=-20, scale=0.4, inner sep=0.3mm]
  \node[circle, draw] {\phantom{$A$}}
  child {node[circle, draw] {$3$}}
  child {node[circle, draw] {$6$}
    child {node[circle, draw] {$1$}}
    child {node[circle, draw] {$2$}}
  }
  ;
\end{tikzpicture}A
\ +\
\begin{tikzpicture}[baseline=-20, scale=0.4, inner sep=0.3mm]
  \node[circle, draw] {\phantom{$A$}}
  child {node[circle, draw] {$3$}}
  child {node[circle, draw] {$6$}
    child {node[circle, draw] {$2$}
      child {node[circle, draw] {$1$}}
    }
  }
  ;
\end{tikzpicture}\,A\,.
\end{equation*}
Note that if $T$ is strictly decreasing ordered and if $i$ is smaller than any
nodes of $T$ then all the trees appearing in this sum are strictly decreasing
ordered.

Applying iteratively this rule, since there is a unique way to get a strictly
decreasing tree adding nodes one by one in the decreasing order, we get the
multiple covariant derivative of $A$:
\begin{equation}\label{multDeriv}
  \left(\prod_{j \in S} \nabla_{\xi_j} \right) A =
  \sum_{T\in\Trees{S}} \nabla_{\xi_{\bullet}}^T A\,.
\end{equation}
With respect to a Sobolev norm $\| \cdot \|_r$ we infer the inequality
\begin{equation}
  \label{estMultDer}
  \left\| \left( \prod_{j \in S} \nabla_{\xi_j} \right) A \right\|_r
  \leqslant C_r^{\left| S \right|}
  \sum_{T\in\Trees{S}} \left\| \nabla^{\nchild{\Root{T}}} A \right\|_r \cdot
  \prod_{\nu \in T^*} \| \nabla^{\nchild{\nu}} \xi_{\nu} \|_r\,.
\end{equation}
We notice indeed the identity $\left| S \right| = \sum_{\nu \in T} \nchild{\nu}$
for any $T\in\Trees{S}$.

\begin{remark}
  The computation made there are very reminiscent to prelie
  computation. Indeed, it is well know that given a flat torsion-free
  connections and $\nabla$ its associated covariant derivative, the bilinear
  operator $X \triangleright Y := \nabla_X Y$ endow the space of vector fields
  with a left pre-Lie algebra structure (see~\cite[Proposition
  3.1]{Manchon}). However in our case, the connection is not flat, so that we
  can't apply the pre-Lie calculus. Still, we can use Cayley trees by ensuring
  that when computing $\nabla_X Y$, the field $X$ is always a single leaf and
  not a proper tree.
\end{remark}

\section{Multiple Lie derivatives of endomorphism sections of the tangent
bundle}

We notice now that for any $A$ smooth endomorphism section of $T_X$ and any
torsion free connection $\nabla$ hold the identity
\begin{equation*}
  L_{\xi} A\ =\ \nabla_{\xi} A + [A, \nabla \xi]\,.
\end{equation*}
We consider the operator $\ad\left( A \right) \assign [A, \bullet]$
acting on endomorphism sections of $T_X$. Then the previous identity rewrites
also as $L_{\xi} = \nabla_{\xi} - \ad\left( \nabla \xi \right)$ over
the space of smooth endomorphism sections of $T_X$. In order to generalize
this fomula to multiple Lie derivatives we need to introduce a few notations.
We denote by $\Part_k$ the set of partitions of the set $[k]$. We notice
that for any $P\in\Part_k$ there exists a unique $p \in P$ such that
$\max p = k$. We denote $p_k$ such $p$. We denote by $P^{\ast} \assign P
\smallsetminus \{p_k \}$. Moreover for any $p \in P$ we denote $p^{\ast}
\assign p \smallsetminus \{\max p\}$. Given a family of smooth vector fields
$\left( \xi_j \right)^k_{j = 1}$ and a subset $S \subset [k]$ we denote by
\begin{equation*}
  \nabla_{\xi_{\bullet}}^S \ \assign \ \prod_{j \in S} \nabla_{\xi_j}\,,
\end{equation*}
where the product is taken in the increasing order from the left to the
right. This notation will only be used in this section.
\begin{lemma}
  Let $\nabla$ be the extension to tensor sections of any torsion free connection. Then for any family of smooth
  vector fields $\left( \xi_j \right)^k_{j = 1}$ the formula holds
  \begin{equation}
    \label{multipLIE}
    \prod_{j = 1}^k L_{\xi_j} =
    \sum_{P \in \Part_{k + 1}}
    (-1)^{|P| - 1}
    \left[ \prod_{p \in P^\ast}
      \ad \left( \nabla_{\xi_\bullet}^{p^\ast}
        \nabla \xi_{\max p} \right)
    \right] \nabla_{\xi_{\bullet}}^{p^{\ast}_{k + 1}}\,,
  \end{equation}
  over the space of smooth endomorphism sections of $T_X$. The product on the
  right hand side is taken in the increasing order provided by the max of the
  elements of $P$, from the left to the right.
\end{lemma}
Before giving the proof, we provide an example of a summand in the
right-hand-side sum. We pick for $k=8$ the set partition
\begin{equation*}
P=\{\{2,3\},\{1,4,6\},\{7\},\{5,8,9\}\}\,.
\end{equation*}
The associated summand is
\begin{equation*}
  (-1)^{3}
  \ad \left( \nabla_{\xi_2} \nabla\xi_3 \right)
  \ad \left( \nabla_{\xi_1}\nabla_{\xi_4} \nabla\xi_6 \right)
  \ad \left( \nabla\xi_7 \right)
  \nabla_{\xi_5}\nabla_{\xi_8}\,.
\end{equation*}
\begin{proof}
  We show the formula (\ref{multipLIE}) by induction on $k$. By the inductive
  assumption
  \begin{equation*}
    \prod_{j = 2}^{k + 1} L_{\xi_j}\ =\
    \sum_{P \in \Part_{2, k + 2}} (-1)^{|P| - 1}
    \left[ \prod_{p \in P^{\ast}}
      \ad \left( \nabla_{\xi_{\bullet}}^{p^{\ast}} \nabla \xi_{\max p}
      \right)
    \right] \nabla_{\xi_{\bullet}}^{p^{\ast}_{k + 2}}\,,
  \end{equation*}
  where $\Part_{2, k + 2}$ denotes the set of partitions of the set
  $\{2, \ldots, k + 2\}$. For any $P \in \Part_{2, k + 2}$ we denote by
  $\Part_P \subset \Part_{k + 2}$ the subset of partitions
  obtained by adding $1$ to one of the parts $p \in P$. We denote by $P_1
  \assign \{\{1\}\} \cup P$. Then
  \begin{align*}
    \nabla_{\xi_1} \prod_{j = 2}^{k + 1} L_{\xi_j}
    &\ =\
    \sum_{P \in \Part_{2, k + 2}}
    (-1)^{|P| - 1}
    \nabla_{\xi_1} \left[ \prod_{p \in P^{\ast}}
      \ad\left(\nabla_{\xi_{\bullet}}^{p^{\ast}} \nabla \xi_{\max p} \right)
    \right]
    \nabla_{\xi_{\bullet}}^{p^{\ast}_{k + 2}} \\
    &\ =\
    \sum_{P \in \Part_{2, k + 2}}
    (-1)^{|P| - 1} \sum_{P' \in \Part_P}
    \left[ \prod_{p' \in P'^{\ast}}
      \ad \left( \nabla_{\xi_{\bullet}}^{p'^{\ast}} \nabla \xi_{\max p'}
      \right)
    \right]
    \nabla_{\xi_{\bullet}}^{p'^{\ast}_{k + 2}}\\
    &\ = \
    \sum_{P \in \Part_{2, k + 2}} \sum_{P' \in \Part_P}
    (-1)^{|P'| - 1}
    \left[ \prod_{p' \in P'^{\ast}}
      \ad \left( \nabla_{\xi_{\bullet}}^{p'^{\ast}} \nabla \xi_{\max p'} \right)
    \right]
    \nabla_{\xi_{\bullet}}^{p'^{\ast}_{k + 2}}\,,
  \end{align*}
  and
  \begin{align*}
    &-\ad \left( \nabla \xi_1 \right) \prod_{j = 2}^{k + 1} L_{\xi_j} \\
    &\ =\ \sum_{P \in \Part_{2, k + 2}} (-1)^{\left| P \right|} \ad \left(
      \nabla \xi_1 \right) \left[ \prod_{p \in P^{\ast}} \ad \left(
        \nabla_{\xi_{\bullet}}^{p^{\ast}} \nabla
        \xi_{\max p} \right) \right] \nabla_{\xi_{\bullet}}^{p^{\ast}_{k + 2}} \\
    &=\ \sum_{P \in \Part_{2, k + 2}} (-1)^{\left| P_1 \right| - 1} \left[
      \prod_{p_1 \in P_1^{\ast}} \ad \left(
        \nabla_{\xi_{\bullet}}^{p_1^{\ast}} \nabla \xi_{\max p_1} \right)
    \right] \nabla_{\xi_{\bullet}}^{p^{\ast}_{1, k + 2}}\,.
  \end{align*}
  The conclusion
  \begin{equation*}
    \prod_{j = 1}^{k + 1} L_{\xi_j}
    \ = \
    \sum_{P \in \Part_{k + 2}} (-1)^{|P| - 1}
    \left[
      \prod_{p \in P^{\ast}}
      \ad \left( \nabla_{\xi_{\bullet}}^{p^{\ast}} \nabla \xi_{\max p} \right)
    \right]
    \nabla_{\xi_{\bullet}}^{p^{\ast}_{k + 2}}\,,
  \end{equation*}
  follows from the fact that
  \begin{equation*}
    \Part_{k + 2}\ =\
    \left[ \bigsqcup_{P \in \Part_{2, k + 2}} \Part_P \right]
    \ \bigsqcup\
    \left[ \bigsqcup_{P \in \Part_{2, k + 2}} \{P_1 \} \right]\,.\qedhere
  \end{equation*}
\end{proof}

Let $\Forests_{k^\circ}$ be the set of forest over $[k]\cup\{\circ\}$.  Recall
that, given a forest $F\in\Forests_{k^\circ}$ we denote $F_\nu$ the subtree rooted at
$\nu$. We also denote by $l_F$ the number of trees of $F$. Recall also that
$F^{\dagger}$ is the forest obtained by truncating from $F$ the tree
$F_\circ$. Then Formula~(\ref{multipLIE}) combined with~(\ref{multDeriv}) implies
\begin{equation}
  \label{multipLIEforest}
  \prod_{j = 1}^k L_{\xi_j}\ =\
  \sum_{F \in \Forests_{k^\circ}} (-1)^{l_F - 1}
  \left[
    \prod_{T \in F^{\dagger}}
    \ad\left( \nabla_{\xi_{\bullet}}^{T} \nabla \xi_{\Root{T}} \right)
  \right]
  \nabla_{\xi_{\bullet}}^{F_\circ}\,,
\end{equation}
where the product is taken in the increasing order provided by the labels of
the roots of $F^{\dagger}$, from the left to the right. Recall that for a tree
$T$, we write $\nu\in T^*$ for $\nu\in T\setminus\{\Root{T}\}$. Then
formula (\ref{multipLIEforest}) and the inequality (\ref{estMultDer}) imply
the estimate
\begin{eqnarray}
  \left\| \left( \prod_{j = 1}^k L_{\xi_j} \right) A \right\|_r & \leqslant &
  C^k_r  \sum_{F \in \Forests_{k^\circ}} 2^{l_F - 1} \left( \left\|
      \nabla^{\nchild{\circ}} A \right\|_r \cdot \prod_{\nu \in F_{\circ}^*} \| \nabla^{\nchild{\nu}} \xi_{\nu} \|_r \right) \times\nonumber
  \\
  \nonumber
  &  & \\
  & \times & \prod_{T \in F^{\dagger}} \left( \left\| \nabla^{1 + \nchild{\Root{T}}}
      \xi_{\Root{T}} \right\|_r \cdot \prod_{\nu \in T^{\ast}} \|
    \nabla^{\nchild{\nu}} \xi_{\nu} \|_r \right) ,\label{estmLieJ}
\end{eqnarray}
where the power of two comes from the estimate of the $\ad$ operator.

\section{Forests and Dyck polynomials}
We define the monomial $X_T$ associated to a tree $T$ as
\begin{equation*}
  X_T \ \assign \ \prod_{\nu \in T} X^{\nchild{\nu} + \delta_{\Root{T},
  \nu}}_{\nu}\,,
\end{equation*}
and the root-truncated monomial $X_T^{\ast}$ as
\begin{equation*}
  X_T^{\ast} \ \assign \ \prod_{\nu \in T^{\ast}} X^{\nchild{\nu}}_{\nu}\,.
\end{equation*}
\begin{example}
  Considering the following tree
  \[T\ =\
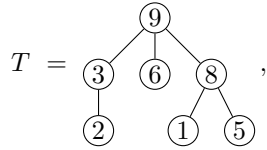
  \begin{tikzpicture}[baseline=-20, scale=0.5, inner sep=0.3mm, minimum size=0.4cm]
  \node[circle, draw] {$9$}
  child {node[circle, draw] {$3$}
    child {node[circle, draw] {$2$}}
  }
  child {node[circle, draw] {$6$}}
  child {node[circle, draw] {$8$}
    child {node[circle, draw] {$1$}}
    child {node[circle, draw] {$5$}}
  };
\end{tikzpicture}\,,
\]
one finds $X_T=X_1^0X_2^0X_3^1X_5^0X_6^0X_8^2X_9^4$ and
$X_T^*=X_1^0X_2^0X_3^1X_5^0X_6^0X_8^2$.

\end{example}
Let $B^l \assign \| \nabla^l A\|_r$ and $X^{\nchild{\nu}}_{\nu} \assign \|
\nabla^{\nchild{\nu}} \xi_{\nu} \|_r$ in the estimate (\ref{estmLieJ}).  The
reader have to be careful that $X_\nu^0= \|\xi_{\nu}\| \neq 1$. Then the sum on
the right hand side of (\ref{estmLieJ}) rewrites as
\begin{equation*}
  \Sigma_k \ \assign \ \sum_{F \in \Forests_{k^\circ}} 2^{l_F - 1}
  B^{\nchild[F]{\circ}}
  X^{\ast}_{F_{\circ}} \prod_{T \in F^{\dagger}} X_T\,.
\end{equation*}
With this notations we state the following theorem.

\begin{theorem}
  \label{TheoLieTree}The sum $\Sigma_k$ satisfies the polynomial expression
  \begin{equation*}
    \Sigma_k \ \equiv \ \Sigma \left( X_1, \ldots, X_k \right) =
    \sum_{P \in \Dyck(k)} C_P B^{D_{P, k}} X^{p_1}_1 \cdots X^{p_k}_k\,.
  \end{equation*}
\end{theorem}
\begin{example}
  We give here the first few values of $\Sigma_k$.
  \begin{align*}
    \Sigma_1 &= B\,X^{(0)} + 2\,X^{(1)}\,, \\
    \Sigma_2 &= B^{2}\,X^{(0,0)} + 2\,B\,X^{(1,0)} + 3\,B\,X^{(0,1)} +
    4\,X^{(1,1)} + 2\,\,X^{(0,2)}\,,\\
    \Sigma_3 &=
    B^{3}\,X^{(0,0,0)} + 2\,B^{2}\,X^{(1,0,0)} + 3\,B^{2}\,X^{(0,2,0)} +
    4\,B\,X^{(1,1,0)} + 2\,B\,X^{(0,1,0)} \\
    &+ 4\,B^{2}\,X^{(0,0,1)} + 6\,B\,X^{(1,0,1)}
    + 9\,B\,X^{(0,1,1)} + 8\,X^{(1,1,1)} + 4\,X^{(0,2,1)} + \\
    &+ 5\,B\,X^{(0,0,2)} + 4\,X^{(1,0,2)} + 6\,X^{(0,1,2)} + 2\,X^{(0,0,3)}\,.
  \end{align*}
\end{example}
From theorem \ref{TheoLieTree} we infer the estimate
\begin{equation}
  \label{multLieDeriv}  \left\| \left( \prod_{j = 1}^k L_{\xi_j} \right) A
  \right\|_r \leqslant C^k_r  \sum_{P \in \tmop{Dyck} \left( k \right)} C_P
  \| \nabla^{D_{P, k}} A\|_r  \prod_{j = 1}^k \| \nabla^{p_j} \xi_j \|_r\,.
\end{equation}
\bigskip

We need a few preliminaries in order to show theorem \ref{TheoLieTree}.
Recall that a Dyck vector of length $k$ is a sequence $P \equiv ( p_1,\dots,
p_k) \in \N^k$ such that
\begin{equation}\label{def-dyck}
  \sum_{1 \leqslant r \leqslant j} p_r \leqslant j\quad \text{ for all $j \in [k]$\,.}
\end{equation}
We denote their set by $\Dyck(k)$. A monomial $X^P=X_1^{p_1}\cdots X_k^{p_k}$
is \emph{Dyck monomial} if $P$ is a Dyck vector.

Recall also that for any $F \in \Forests_{k^\circ}$ we denote by $F_\nu$ the tree
rooted at $\nu$ and by $F^\dagger$ the forest $F \smallsetminus \{F_\circ
\}$.  For any forest $F \in \Forests_{k^\circ}$ we define
\begin{equation*}
  X_F \ \assign\ X^{\ast}_{F_\circ} \prod_{T \in F^\dagger} X_T\,.
\end{equation*}
See Example~\ref{forest-prime} below. We observe that $X_F$ is always a Dyck
monomial $X^P$. Indeed let
\begin{equation*}
  p_j = \nchild[F]j + \delta_{j,\operatorname{Root}(F)} \,.
\end{equation*}
for all $j = 1, \ldots, k$ where $\delta_{j,\operatorname{Root}(F)}$ is $1$ if
$j$ is a root of $F$ and $0$ otherwise.
Then Condition~(\ref{def-dyck}) follows from the obvious identity
\begin{align*}
  \sum_{1 \leqslant r \leqslant j} p_r
  &\ =\  \left| \{\nu \in F \cap [j] \mid
  \nu \text{ child in } F \cap [j]\} \right|\\
  &\ +\ \left| \{\nu \in F \cap [j] \mid \nu \text{ root in } F\} \right|,
\end{align*}
and the fact that
\begin{equation*}
  \{\nu \in F \cap [j] \mid \nu \text{ child in } F \cap [j]\}\ \sqcup\
  \{\nu \in F \cap [j] \mid \nu \text{ root in } F\} \ \subseteq \ [j]\,.
\end{equation*}
For any $P \in \Dyck(k)$, we define the set
\begin{equation*}
  \Forests_P\ \assign\ \left\{ F \in \Forests_{k^\circ} \mid X_F = X^P \right\}.
\end{equation*}

\begin{example}
We fix $P=(0, 0, 2, 1, 1)$. The following picture show all the forests
such that $X_F=X^P$, sorted according to their length.
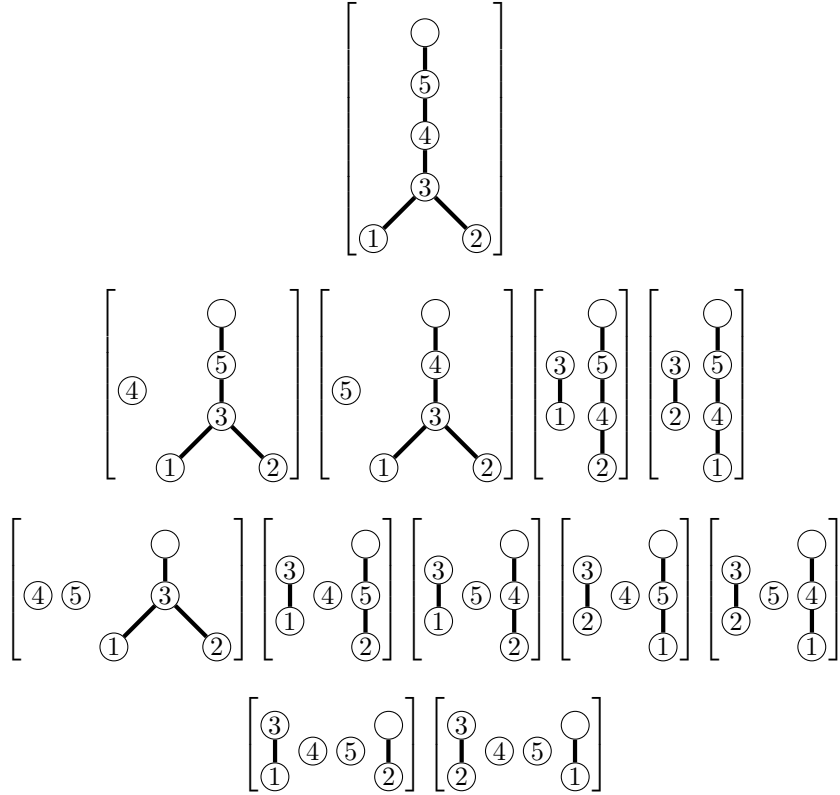
\begin{gather*}
\left[{
\newcommand{\nodea}{\node[draw,circle] (a) {\phantom{$6$}};}
\newcommand{\nodeb}{\node[draw,circle] (b) {$5$};}
\newcommand{\nodec}{\node[draw,circle] (c) {$4$};}
\newcommand{\noded}{\node[draw,circle] (d) {$3$};}
\newcommand{\nodee}{\node[draw,circle] (e) {$1$};}
\newcommand{\nodef}{\node[draw,circle] (f) {$2$};}
\begin{tikzpicture}[baseline=0, inner sep=0.3mm]
\matrix[column sep=.3cm, row sep=.3cm,ampersand replacement=\&]{
         \& \nodea  \&         \\ 
         \& \nodeb  \&         \\ 
         \& \nodec  \&         \\ 
         \& \noded  \&         \\ 
 \nodee  \&         \& \nodef  \\
};
\path[ultra thick] (d) edge (e) edge (f)
	(c) edge (d)
	(b) edge (c)
	(a) edge (b);
\end{tikzpicture}}\right]
\\[3mm] 
\left[{ \newcommand{\nodea}{\node[draw,circle] (a) {$4$}
;}\begin{tikzpicture}[baseline=0, inner sep=0.3mm]
\matrix[column sep=.3cm, row sep=.3cm,ampersand replacement=\&]{
 \nodea  \\
};
\end{tikzpicture}}\,{ \newcommand{\nodea}{\node[draw,circle] (a) {\phantom{$6$}}
;}\newcommand{\nodeb}{\node[draw,circle] (b) {$5$}
;}\newcommand{\nodec}{\node[draw,circle] (c) {$3$}
;}\newcommand{\noded}{\node[draw,circle] (d) {$1$}
;}\newcommand{\nodee}{\node[draw,circle] (e) {$2$}
;}\begin{tikzpicture}[baseline=0, inner sep=0.3mm]
\matrix[column sep=.3cm, row sep=.3cm,ampersand replacement=\&]{
         \& \nodea  \&         \\ 
         \& \nodeb  \&         \\ 
         \& \nodec  \&         \\ 
 \noded  \&         \& \nodee  \\
};
\path[ultra thick] (c) edge (d) edge (e)
	(b) edge (c)
	(a) edge (b);
\end{tikzpicture}}\right]
\left[{ \newcommand{\nodea}{\node[draw,circle] (a) {$5$}
;}\begin{tikzpicture}[baseline=0, inner sep=0.3mm]
\matrix[column sep=.3cm, row sep=.3cm,ampersand replacement=\&]{
 \nodea  \\
};
\end{tikzpicture}}\,{ \newcommand{\nodea}{\node[draw,circle] (a) {\phantom{$6$}}
;}\newcommand{\nodeb}{\node[draw,circle] (b) {$4$}
;}\newcommand{\nodec}{\node[draw,circle] (c) {$3$}
;}\newcommand{\noded}{\node[draw,circle] (d) {$1$}
;}\newcommand{\nodee}{\node[draw,circle] (e) {$2$}
;}\begin{tikzpicture}[baseline=0, inner sep=0.3mm]
\matrix[column sep=.3cm, row sep=.3cm,ampersand replacement=\&]{
         \& \nodea  \&         \\ 
         \& \nodeb  \&         \\ 
         \& \nodec  \&         \\ 
 \noded  \&         \& \nodee  \\
};
\path[ultra thick] (c) edge (d) edge (e)
	(b) edge (c)
	(a) edge (b);
\end{tikzpicture}}\right]
\left[{ \newcommand{\nodea}{\node[draw,circle] (a) {$3$}
;}\newcommand{\nodeb}{\node[draw,circle] (b) {$1$}
;}\begin{tikzpicture}[baseline=0, inner sep=0.3mm]
\matrix[column sep=.3cm, row sep=.3cm,ampersand replacement=\&]{
 \nodea  \\ 
 \nodeb  \\
};
\path[ultra thick] (a) edge (b);
\end{tikzpicture}}\ { \newcommand{\nodea}{\node[draw,circle] (a) {\phantom{$6$}}
;}\newcommand{\nodeb}{\node[draw,circle] (b) {$5$}
;}\newcommand{\nodec}{\node[draw,circle] (c) {$4$}
;}\newcommand{\noded}{\node[draw,circle] (d) {$2$}
;}\begin{tikzpicture}[baseline=0, inner sep=0.3mm]
\matrix[column sep=.3cm, row sep=.3cm,ampersand replacement=\&]{
 \nodea  \\ 
 \nodeb  \\ 
 \nodec  \\ 
 \noded  \\
};
\path[ultra thick] (c) edge (d)
	(b) edge (c)
	(a) edge (b);
\end{tikzpicture}}\right]
\left[{ \newcommand{\nodea}{\node[draw,circle] (a) {$3$}
;}\newcommand{\nodeb}{\node[draw,circle] (b) {$2$}
;}\begin{tikzpicture}[baseline=0, inner sep=0.3mm]
\matrix[column sep=.3cm, row sep=.3cm,ampersand replacement=\&]{
 \nodea  \\ 
 \nodeb  \\
};
\path[ultra thick] (a) edge (b);
\end{tikzpicture}}\ { \newcommand{\nodea}{\node[draw,circle] (a) {\phantom{$6$}}
;}\newcommand{\nodeb}{\node[draw,circle] (b) {$5$}
;}\newcommand{\nodec}{\node[draw,circle] (c) {$4$}
;}\newcommand{\noded}{\node[draw,circle] (d) {$1$}
;}\begin{tikzpicture}[baseline=0, inner sep=0.3mm]
\matrix[column sep=.3cm, row sep=.3cm,ampersand replacement=\&]{
 \nodea  \\ 
 \nodeb  \\ 
 \nodec  \\ 
 \noded  \\
};
\path[ultra thick] (c) edge (d)
	(b) edge (c)
	(a) edge (b);
\end{tikzpicture}}\right]
\\[3mm] 
\left[{ \newcommand{\nodea}{\node[draw,circle] (a) {$4$}
;}\begin{tikzpicture}[baseline=0, inner sep=0.3mm]
\matrix[column sep=.3cm, row sep=.3cm,ampersand replacement=\&]{
 \nodea  \\
};
\end{tikzpicture}}\,{ \newcommand{\nodea}{\node[draw,circle] (a) {$5$}
;}\begin{tikzpicture}[baseline=0, inner sep=0.3mm]
\matrix[column sep=.3cm, row sep=.3cm,ampersand replacement=\&]{
 \nodea  \\
};
\end{tikzpicture}}\,{ \newcommand{\nodea}{\node[draw,circle] (a) {\phantom{$6$}}
;}\newcommand{\nodeb}{\node[draw,circle] (b) {$3$}
;}\newcommand{\nodec}{\node[draw,circle] (c) {$1$}
;}\newcommand{\noded}{\node[draw,circle] (d) {$2$}
;}\begin{tikzpicture}[baseline=0, inner sep=0.3mm]
\matrix[column sep=.3cm, row sep=.3cm,ampersand replacement=\&]{
         \& \nodea  \&         \\ 
         \& \nodeb  \&         \\ 
 \nodec  \&         \& \noded  \\
};
\path[ultra thick] (b) edge (c) edge (d)
	(a) edge (b);
\end{tikzpicture}}\right]
\left[{ \newcommand{\nodea}{\node[draw,circle] (a) {$3$}
;}\newcommand{\nodeb}{\node[draw,circle] (b) {$1$}
;}\begin{tikzpicture}[baseline=0, inner sep=0.3mm]
\matrix[column sep=.3cm, row sep=.3cm,ampersand replacement=\&]{
 \nodea  \\ 
 \nodeb  \\
};
\path[ultra thick] (a) edge (b);
\end{tikzpicture}}\,{ \newcommand{\nodea}{\node[draw,circle] (a) {$4$}
;}\begin{tikzpicture}[baseline=0, inner sep=0.3mm]
\matrix[column sep=.3cm, row sep=.3cm,ampersand replacement=\&]{
 \nodea  \\
};
\end{tikzpicture}}\,{ \newcommand{\nodea}{\node[draw,circle] (a) {\phantom{$6$}}
;}\newcommand{\nodeb}{\node[draw,circle] (b) {$5$}
;}\newcommand{\nodec}{\node[draw,circle] (c) {$2$}
;}\begin{tikzpicture}[baseline=0, inner sep=0.3mm]
\matrix[column sep=.3cm, row sep=.3cm,ampersand replacement=\&]{
 \nodea  \\ 
 \nodeb  \\ 
 \nodec  \\
};
\path[ultra thick] (b) edge (c)
	(a) edge (b);
\end{tikzpicture}}\right]
\left[{ \newcommand{\nodea}{\node[draw,circle] (a) {$3$}
;}\newcommand{\nodeb}{\node[draw,circle] (b) {$1$}
;}\begin{tikzpicture}[baseline=0, inner sep=0.3mm]
\matrix[column sep=.3cm, row sep=.3cm,ampersand replacement=\&]{
 \nodea  \\ 
 \nodeb  \\
};
\path[ultra thick] (a) edge (b);
\end{tikzpicture}}\,{ \newcommand{\nodea}{\node[draw,circle] (a) {$5$}
;}\begin{tikzpicture}[baseline=0, inner sep=0.3mm]
\matrix[column sep=.3cm, row sep=.3cm,ampersand replacement=\&]{
 \nodea  \\
};
\end{tikzpicture}}\,{ \newcommand{\nodea}{\node[draw,circle] (a) {\phantom{$6$}}
;}\newcommand{\nodeb}{\node[draw,circle] (b) {$4$}
;}\newcommand{\nodec}{\node[draw,circle] (c) {$2$}
;}\begin{tikzpicture}[baseline=0, inner sep=0.3mm]
\matrix[column sep=.3cm, row sep=.3cm,ampersand replacement=\&]{
 \nodea  \\ 
 \nodeb  \\ 
 \nodec  \\
};
\path[ultra thick] (b) edge (c)
	(a) edge (b);
\end{tikzpicture}}\right]
\left[{ \newcommand{\nodea}{\node[draw,circle] (a) {$3$}
;}\newcommand{\nodeb}{\node[draw,circle] (b) {$2$}
;}\begin{tikzpicture}[baseline=0, inner sep=0.3mm]
\matrix[column sep=.3cm, row sep=.3cm,ampersand replacement=\&]{
 \nodea  \\ 
 \nodeb  \\
};
\path[ultra thick] (a) edge (b);
\end{tikzpicture}}\,{ \newcommand{\nodea}{\node[draw,circle] (a) {$4$}
;}\begin{tikzpicture}[baseline=0, inner sep=0.3mm]
\matrix[column sep=.3cm, row sep=.3cm,ampersand replacement=\&]{
 \nodea  \\
};
\end{tikzpicture}}\,{ \newcommand{\nodea}{\node[draw,circle] (a) {\phantom{$6$}}
;}\newcommand{\nodeb}{\node[draw,circle] (b) {$5$}
;}\newcommand{\nodec}{\node[draw,circle] (c) {$1$}
;}\begin{tikzpicture}[baseline=0, inner sep=0.3mm]
\matrix[column sep=.3cm, row sep=.3cm,ampersand replacement=\&]{
 \nodea  \\ 
 \nodeb  \\ 
 \nodec  \\
};
\path[ultra thick] (b) edge (c)
	(a) edge (b);
\end{tikzpicture}}\right]
\left[{ \newcommand{\nodea}{\node[draw,circle] (a) {$3$}
;}\newcommand{\nodeb}{\node[draw,circle] (b) {$2$}
;}\begin{tikzpicture}[baseline=0, inner sep=0.3mm]
\matrix[column sep=.3cm, row sep=.3cm,ampersand replacement=\&]{
 \nodea  \\ 
 \nodeb  \\
};
\path[ultra thick] (a) edge (b);
\end{tikzpicture}}\,{ \newcommand{\nodea}{\node[draw,circle] (a) {$5$}
;}\begin{tikzpicture}[baseline=0, inner sep=0.3mm]
\matrix[column sep=.3cm, row sep=.3cm,ampersand replacement=\&]{
 \nodea  \\
};
\end{tikzpicture}}\,{ \newcommand{\nodea}{\node[draw,circle] (a) {\phantom{$6$}}
;}\newcommand{\nodeb}{\node[draw,circle] (b) {$4$}
;}\newcommand{\nodec}{\node[draw,circle] (c) {$1$}
;}\begin{tikzpicture}[baseline=0, inner sep=0.3mm]
\matrix[column sep=.3cm, row sep=.3cm,ampersand replacement=\&]{
 \nodea  \\ 
 \nodeb  \\ 
 \nodec  \\
};
\path[ultra thick] (b) edge (c)
	(a) edge (b);
\end{tikzpicture}}\right]
\\[3mm] 
\left[{ \newcommand{\nodea}{\node[draw,circle] (a) {$3$}
;}\newcommand{\nodeb}{\node[draw,circle] (b) {$1$}
;}\begin{tikzpicture}[baseline=0, inner sep=0.3mm]
\matrix[column sep=.3cm, row sep=.3cm,ampersand replacement=\&]{
 \nodea  \\ 
 \nodeb  \\
};
\path[ultra thick] (a) edge (b);
\end{tikzpicture}}\,{ \newcommand{\nodea}{\node[draw,circle] (a) {$4$}
;}\begin{tikzpicture}[baseline=0, inner sep=0.3mm]
\matrix[column sep=.3cm, row sep=.3cm,ampersand replacement=\&]{
 \nodea  \\
};
\end{tikzpicture}}\,{ \newcommand{\nodea}{\node[draw,circle] (a) {$5$}
;}\begin{tikzpicture}[baseline=0, inner sep=0.3mm]
\matrix[column sep=.3cm, row sep=.3cm,ampersand replacement=\&]{
 \nodea  \\
};
\end{tikzpicture}}\,{ \newcommand{\nodea}{\node[draw,circle] (a) {\phantom{$6$}}
;}\newcommand{\nodeb}{\node[draw,circle] (b) {$2$}
;}\begin{tikzpicture}[baseline=0, inner sep=0.3mm]
\matrix[column sep=.3cm, row sep=.3cm,ampersand replacement=\&]{
 \nodea  \\ 
 \nodeb  \\
};
\path[ultra thick] (a) edge (b);
\end{tikzpicture}}\right]
\left[{ \newcommand{\nodea}{\node[draw,circle] (a) {$3$}
;}\newcommand{\nodeb}{\node[draw,circle] (b) {$2$}
;}\begin{tikzpicture}[baseline=0, inner sep=0.3mm]
\matrix[column sep=.3cm, row sep=.3cm,ampersand replacement=\&]{
 \nodea  \\ 
 \nodeb  \\
};
\path[ultra thick] (a) edge (b);
\end{tikzpicture}}\,{ \newcommand{\nodea}{\node[draw,circle] (a) {$4$}
;}\begin{tikzpicture}[baseline=0, inner sep=0.3mm]
\matrix[column sep=.3cm, row sep=.3cm,ampersand replacement=\&]{
 \nodea  \\
};
\end{tikzpicture}}\,{ \newcommand{\nodea}{\node[draw,circle] (a) {$5$}
;}\begin{tikzpicture}[baseline=0, inner sep=0.3mm]
\matrix[column sep=.3cm, row sep=.3cm,ampersand replacement=\&]{
 \nodea  \\
};
\end{tikzpicture}}\,{ \newcommand{\nodea}{\node[draw,circle] (a) {\phantom{$6$}}
;}\newcommand{\nodeb}{\node[draw,circle] (b) {$1$}
;}\begin{tikzpicture}[baseline=0, inner sep=0.3mm]
\matrix[column sep=.3cm, row sep=.3cm,ampersand replacement=\&]{
 \nodea  \\ 
 \nodeb  \\
};
\path[ultra thick] (a) edge (b);
\end{tikzpicture}}\right]
\end{gather*}
As a result we get that
\[C_{(0, 0, 2, 1, 1)} = 1 + 4\cdot2 + 5\cdot2^2 + 2\cdot2^3 = 45\,.\]
This agrees with
\begin{multline*}
  C_{(0, 0, 2, 1, 1)} =
  \left[2\binom{0}{-1}+\binom{0}{0}\right]
  \left[2\binom{1}{-1}+\binom{1}{0}\right]
  \ \times\\  \times
  \left[2\binom{2}{1}+\binom{2}{2}\right]
  \left[2\binom{1}{0}+\binom{1}{1}\right]
  \left[2\binom{1}{0}+\binom{1}{1}\right]\,.
\end{multline*}
\end{example}

\begin{definition}
  For any $F \in \Forests_{k^\circ}$, $k > 0$ we denote by $F' \in
  \Forests_{k}$ the forest obtained by pruning the childrens of the node
  labeled $\circ$ and grafting them, to the node labeled $k$. Of course, the
  old and new children of $k$ are shuffled to draw them in increasing
  order. Finally by replacing $k$ by $\circ$ in $F'$, we consider that $F'$
  actually belongs to $\Forests_{k-1^\circ}$.
\end{definition}
\begin{example}\label{forest-prime}
  It will becomes apparent in the following proofs that there are two
  different cases, whether $k$ is a child of $\circ$ or not.

  \begin{itemize}
  \item We start by a case where $k=9$ is a not child of $\circ$ and thus a
    root of $F$.  We show below a forest $F$ together with its associated
    $F'$.
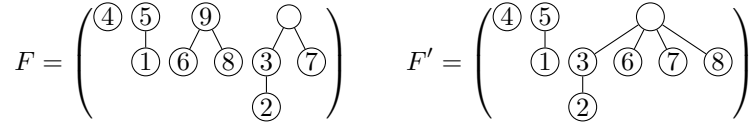
    \[
    F = \left(
      \begin{tikzpicture}[baseline=-20, scale=0.4, inner sep=0.3mm]
        \node[circle, draw] {$4$};
      \end{tikzpicture}\
      \begin{tikzpicture}[baseline=-20, scale=0.4, inner sep=0.3mm]
        \node[circle, draw] {$5$}
        child {node[circle, draw] {$1$}};
      \end{tikzpicture}\
      \begin{tikzpicture}[baseline=-20, scale=0.4, inner sep=0.3mm]
        \node[circle, draw] {$9$}
        child {node[circle, draw] {$6$}}
        child {node[circle, draw] {$8$}};
      \end{tikzpicture}\
      \begin{tikzpicture}[baseline=-20, scale=0.4, inner sep=0.3mm]
        \node[circle, draw] {\phantom{$\circ$}}
        child {node[circle, draw] {$3$}
          child {node[circle, draw] {$2$}}
        }
        child {node[circle, draw] {$7$}};
      \end{tikzpicture}\right)
    \qquad
    F' = \left(
      \begin{tikzpicture}[baseline=-20, scale=0.4, inner sep=0.3mm]
        \node[circle, draw] {$4$};
      \end{tikzpicture}\
      \begin{tikzpicture}[baseline=-20, scale=0.4, inner sep=0.3mm]
        \node[circle, draw] {$5$}
        child {node[circle, draw] {$1$}};
      \end{tikzpicture}\
      \begin{tikzpicture}[baseline=-20, scale=0.4, inner sep=0.3mm]
        \node[circle, draw] {\phantom{$9$}}
        child {node[circle, draw] {$3$}
          child {node[circle, draw] {$2$}}
        }
        child {node[circle, draw] {$6$}}
        child {node[circle, draw] {$7$}}
        child {node[circle, draw] {$8$}};
      \end{tikzpicture}\right)
    \]
    One can check that their associated monomials are
    $X_F=X^{(0,0,1,1,2,0,0,0,3)}$ and $X_{F'}=X^{(0,0,1,1,2,0,0,0)}$.
  \item We show now a forest $F$ where $k=9$ is a child of $\circ$ together
    with its associated $F'$.
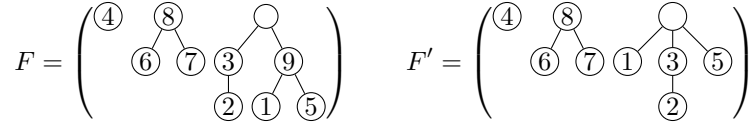
    \[
    F = \left(
      \begin{tikzpicture}[baseline=-20, scale=0.4, inner sep=0.3mm]
        \node[circle, draw] {$4$};
      \end{tikzpicture}\
      \begin{tikzpicture}[baseline=-20, scale=0.4, inner sep=0.3mm]
        \node[circle, draw] {$8$}
        child {node[circle, draw] {$6$}}
        child {node[circle, draw] {$7$}};
      \end{tikzpicture}\
      \begin{tikzpicture}[baseline=-20, scale=0.4, inner sep=0.3mm]
        \node[circle, draw] {\phantom{$\circ$}}
        child [sibling distance=2.5cm] {node[circle, draw] {$3$}
          child {node[circle, draw] {$2$}}
        }
        child {node[circle, draw] {$9$}
          child {node[circle, draw] {$1$}}
          child {node[circle, draw] {$5$}}
        };
      \end{tikzpicture}\right)
    \qquad
    F' = \left(
      \begin{tikzpicture}[baseline=-20, scale=0.4, inner sep=0.3mm]
        \node[circle, draw] {$4$};
      \end{tikzpicture}\
      \begin{tikzpicture}[baseline=-20, scale=0.4, inner sep=0.3mm]
        \node[circle, draw] {$8$}
        child {node[circle, draw] {$6$}}
        child {node[circle, draw] {$7$}};
      \end{tikzpicture}\
      \begin{tikzpicture}[baseline=-20, scale=0.4, inner sep=0.3mm]
        \node[circle, draw] {\phantom{$9$}}
        child {node[circle, draw] {$1$}}
        child {node[circle, draw] {$3$}
          child {node[circle, draw] {$2$}}
        }
        child {node[circle, draw] {$5$}};
      \end{tikzpicture}\right)
    \]
    One can check that their associated monomials are
    $X_F=X^{(0,0,1,1,0,0,0,3,2)}$ and $X_{F'}=X^{(0,0,1,1,0,0,0,3)}$.
  \end{itemize}
\end{example}
\begin{lemma} \label{IdMon}
  For any $F \in \Forests_{k^\circ}$, with $k \geqslant 0$
  the identity $X_{F'} = X_{F \mid X_k = 1}$ holds.
\end{lemma}

\begin{proof}
  By definition of $F'$ we infer the identity $\nchild[F']{j} = \nchild[F]{j}$ for any node
  labeled $j < k$. Moreover the node labeled $j < k$ is a root in $F'$ if and
  only if it is a root in $F$.
\end{proof}
For $P=(p_1,\dots,p_k)$, we denote by $\deg \left( X^P \right) \assign
p_1+\dots+p_k$.

\begin{lemma}
  \label{degId}For any $F \in \Forests_{k^\circ}$, with $k \geqslant 0$ the
  identity $\nchild[F]{\circ} = k - \deg \left( X_F \right)$ holds. In
  particular $\nchild[F]{\circ}$ depends only on $X_F$ and not on $F$.
\end{lemma}

\begin{proof}
  We prove this by induction on $k$. If $k = 0$, there is only one forest $F
  \in \Forests_{\{\circ\}}$ whose monomial is $X_F = 1$. We consider now a
  forest $F \in \Forests_{k^\circ}$ with $X_F = X^{p_1}_1 \cdots
  X^{p_k}_k$. We reccal that the power $p_j$ of a monomial $X^{p_j}_j$
  satisfies: $p_j = \nchild[F]j + 1$ if $j = \Root{T}$ for some $T \in F$ and
  $p_j = \nchild[F]j$ if $j \neq \Root{T}$ for all $T \in F$. By the inductive
  assumption the statement hold for $F'$.  There are two case, whether $k$ is
  a child of $\circ$ or a root in $F$.
  \begin{itemize}
  \item If $k$ is not a child of $\circ$, i.e. a root, then
    $\nchild[F']k = \nchild[F]k + \nchild[F]{\circ}$. Therefore
    \begin{align*}
      \nchild[F]{\circ} & = \nchild[F']k - \nchild[F]k\\
      & = k - 1 - \deg \left( X_{F'} \right) - \left( p_k - 1 \right)\\
      & = k - \deg \left( X_F \right)\,.
    \end{align*}
    The last equality follows from Lemma~\ref{IdMon}.
  \item If $k$ is a child of $\circ$, by definition of $F'$, the childrens of
    $k$ in $F'$ are the union of those of $k$ in $F$ and those of $\circ$
    exept $k$. Thus $\nchild[F']{k} = \nchild[F]{k} + \nchild[F]{\circ} -
    1$. We infer
    \begin{align*}
      \nchild[F]{\circ} & = 1 + \nchild[F']k - \nchild[F]k\\
      & = 1 + \left( k - 1 - \deg \left( X_{F'} \right) \right) - p_k\\
      & = k - \deg \left( X_F \right)\,.
    \end{align*}
    The last equality follows from Lemma~\ref{IdMon}. \qedhere
  \end{itemize}
\end{proof}
\bigskip

\begin{proof}[Proof of theorem \ref{TheoLieTree}.]
  For any $P\in\Dyck(k)$, define
  \begin{equation*}
    C'_P\ \assign\ \sum_{F \in \Forests_P} 2^{l_F - 1}\,.
  \end{equation*}
  Then
  \begin{equation*}
    \Sigma_k
    \ =\ \sum_{F \in \Forests_{k^\circ}} 2^{l_F - 1} B^{\nchild[F]{\circ}} X_F
    \ =\ \sum_{P \in \Dyck(k)} C'_P B^{D_{P, k}} X^P,
  \end{equation*}
  thanks to lemma \ref{degId}. Thus our goal is to show the equality
  \begin{equation}
    C'_P \ =\  \prod_{j = 1}^k
    \left[ 2 \binom{D_{P, j - 1}}{p_j - 1} + \binom{D_{P, j - 1}}{p_j} \right]
    \ =\ C_P \,.
  \end{equation}
  We proceed by induction on $k$. We decompose $\Forests_P$ as a disjoint
  union as
  \begin{align*}
    \Forests_P&\ =\
    \bigsqcup_{H \in \Forests_{p_1, \ldots p_{k - 1}}} \Forests_P(H)\,,\\
    \intertext{where}
    \Forests_P(H)&\ \assign\ \left\{
      F \in \Forests_{P} \mid F' = H
    \right\}\,.
  \end{align*}
  As a consequence,
  \begin{equation}\label{expansion-Cpprime}
    C'_P = \sum_{H \in \Forests_{p_1, \ldots p_{k - 1}}}\
      \sum_{F\in\Forests_P(H)} 2^{l_F-1}\,.
  \end{equation}
  Thanks to Lemma~\ref{degId}, we infer $\nchild[H]k = k-1 - \deg(X_H)$. By
  Lemma~\ref{IdMon}, we have $\deg(X_H)=p_1+\dots+p_{k-1}$ so that
  \begin{equation}\label{rel_H_Dpk1}
    \nchild[H]k = D_{P,k-1}\,.
  \end{equation}
  We now distinguish two cases, whether $k$ is a root of $F\in\Forests_P(H)$
  or not:
  \begin{itemize}
  \item The number of $F\in\Forests_P(H)$ such that $k$ is a root in $F$ is
    given by
    \begin{equation*}
      \binom{D_{P, k - 1}}{p_k - 1}\,.
    \end{equation*}
  Indeed for fixed $H$, the equality
  \begin{equation}
    \label{trunkForrestEq} F \cap [k - 1] = H \cap [k - 1]\,,
  \end{equation}
  shows that the only freedom of the forests $F\in\Forests_P(H)$ which satisfy
  (\ref{trunkForrestEq}) is in the choice of the the $\nchild[F]k=p_k - 1$
  childrens of $k$ in $F$ among the $\nchild[H]k=D_{P, k - 1}$ childrens of $k$ in
  $H$. These childrens are given by the union of the childrens of $\circ$ and
  $k$ in $F$.  The case $p_k=0$ does not occur since $k$ is a root in
  $F$. This is consistent with the convention $\binom{a}{-1}=0$.

  Using the fact that, in this case, $l_F= l_{H}+1$, one conclude that for any
  fixed $H \in \Forests_{p_1, \ldots p_{k-1}}$, one has
  \begin{equation}\label{case-k-root}
    \sum_{\substack{F\in\Forests_P(H)\\
      \text{$k$ is a root of $F$}}} 2^{l_F-1}
    \ =\
    2\cdot2^{l_{H}-1}\binom{D_{P, k - 1}}{p_k - 1}\,.
  \end{equation}

\item The number of $F\in\Forests_P(H)$ such that $k$ is not a root in $F$
  is given by
  \[
  \binom{D_{P, k - 1}}{p_k}\,.
  \]
  The reason is the same as before. We notice that the definition of Dyck
  vector allows the case were $p_k = D_{P, k - 1}+1$ (this is the case in
  Example~\ref{ex-comput-Cp} for $k=5$). But $k$ has $D_{P, k - 1}$ children
  in $F'$ thanks to~(\ref{rel_H_Dpk1}). Therefore it cannot have $p_k$
  children in $F$. This is consistent with the convention $\binom{a}{a+1}=0$.

  Using the fact that, in this case, $l_F= l_{H}$, one conclude that for any
  fixed $H \in \Forests_{p_1, \ldots p_{k-1}}$, one has
  \begin{equation}\label{case-k-not-root}
    \sum_{\substack{F\in\Forests_P(H)\\
      \text{$k$ is not a root of $F$}}} 2^{l_F-1}
    \ =\
    2^{l_{H}-1}\binom{D_{P, k - 1}}{p_k}\,.
  \end{equation}
  \end{itemize}
  Combining Equation~(\ref{expansion-Cpprime}) with the two
  identities~(\ref{case-k-root}) and (\ref{case-k-not-root}) we obtain
  \begin{align*}
    \label{inductiveFormul}
    C'_P\ &=\
    \sum_{H \in \Forests_{p_1, \ldots p_{k - 1}}}
    \left[
      2 \binom{D_{P, k - 1}}{p_k - 1} + \binom{D_{P, k - 1}}{p_k}
    \right] 2^{l_{H} - 1} \\
    &\ =
    C'_{p_1, \ldots p_{k - 1}}
    \left[
      2 \binom{D_{P, k - 1}}{p_k - 1} + \binom{D_{P, k - 1}}{p_k}
    \right]\,.
  \end{align*}
  We finally infer the required identity $C'_P = C_P$ by induction on $k$.
\end{proof}
\begin{remark}
We notice the formula
\begin{equation*}
  C_P \ =\
  \prod_{\substack{
      1 \leqslant j \leqslant k\\
      p_j  \neq 0
    }}
  \left( 2 + \frac{D_{P, j}}{p_j} \right)  \binom{D_{P, j - 1}}{p_j - 1}\,.
\end{equation*}
Indeed
\begin{equation*}
  C_P \ = \
  \prod_{\substack{
      1 \leqslant j \leqslant k\\
      p_j  \neq 0
    }}
  \left[ 2 \binom{D_{P, j - 1}}{p_j - 1} + \binom{D_{P, j - 1}}{p_j} \right],
\end{equation*}
and
\begin{equation*}
  \binom{D_{P, j - 1}}{p_j} \ = \
  \frac{D_{P, j - 1} - p_j + 1}{p_j}
  \binom{D_{P, j - 1}}{p_j - 1}\,.
\end{equation*}
Then the conclusion follows from the identity $D_{P, j} = D_{P, j - 1} - p_j +
1$.
\end{remark}

\section{Higher order covariant derivatives of tensors}

In sequel, we denote for any $S\subset\N_{>0}$ we denote
\begin{equation*}
  \nabla^{S}_{\xi_{\bullet}}
  \ \assign\
  \left(
    \bigotimes_{p \in S} \xi_p
  \right) \neg \nabla^{|S|}\,.
\end{equation*}
The reader should not confuse this with $\nabla^{T}_{\xi_{\bullet}}$ which is
used when $T$ is a tree.

We set $\Map{h}{l} \assign \{ \mu : [h] \longrightarrow [l]\}$.
Let $A_j$ be smooth sections of the bundle $\left( T^{\ast}_X \right)^{\otimes
q_j} \otimes T_X$, $j = 1, \ldots, l$. There are many situations in which the
notion of product $\prod_{j = 1}^l A_j$ is well defined. This is the case for
instance when:
\begin{enumerate}
\item\label{case-prod-comp} $q_j = 1$ for all $j = 1, \ldots, l - 1$. In
  this case the product is just a composition $A_1\circ A_2\circ\dots\circ
  A_l(\xi_1\otimes\dots\otimes \xi_{q_l})$ of endomorphisms with a
  $q_l$-linear map giving a $q_l$-linear map.
\item\label{case-prod-tree} $q_1 = l - 1$ and $q_j = 0, 1$, for $j
  \geqslant 2$. In this second case, the product is the application of a
  $(l-1)$-linear map $A_1(A_2[\xi_2]\otimes A_3[\xi_3] \otimes \dots \otimes
  A_l[\xi_l])$ to either vector fields $A_j$ when $q_j=0$ or the value
  $A_j(\xi_l)$ of the linear map $A_j$ when $q_j=1$. The bracket around the
  $[\xi_j]$ means that they are only present if $q_j=1$. The result is a
  $q_2+q_3+\dots+q_l$-linear map.
\end{enumerate}
In all these cases the following lemma hold.
\begin{lemma}
  \label{abstrLeibnitz}Let $A_j$ be smooth sections of $\left( T^{\ast}_X
  \right)^{\otimes q_j} \otimes T_X$, $j = 1, \ldots, l$ such that the formal
  product $\prod_{j = 1}^l A_j$ is well defined and let $\left( \xi_p
  \right)^h_{p = 1}$ be a family of vector fields over $X$. Then the $h$-order
  covariant derivative satisfies the general Leibnitz identity
  \begin{equation*}
    \left( \bigotimes_{p = 1}^h \xi_p \right)
    \neg \nabla^h \left( \prod_{j=1}^l A_j \right)
    \ = \ 
    \sum_{\mu \in \Map{h}{l}}
    \left(\prod_{j = 1}^l
      \nabla^{\mu^{-1}(j)}_{\xi_{\bullet}} A_j\right).
  \end{equation*}
\end{lemma}

\begin{proof}
  We proceed by induction. We remind first the inductive definition of higher
  order covariant derivative:
  \begin{equation*}
    \nabla_{\xi_0\otimes\dots\otimes\xi_h}^{h + 1}
    \ \assign \
    \nabla_{\xi_0}
    \nabla_{\xi_1\otimes\dots\otimes\xi_h}^h -
    \sum_{p = 1}^h
    \nabla_{\xi_1\otimes\dots\otimes\nabla_{\xi_0} \xi_p\otimes\dots\otimes\xi_h}^h\,.
  \end{equation*}
  Taking a covariant derivative of the inductive assumption we infer
  \begin{multline*}
    \nabla_{\xi_0}
    \left[
      \left( \bigotimes_{p = 1}^h \xi_p \right) \neg
      \nabla^h \left( \prod_{j = 1}^l A_j \right)
    \right]\\
     =\ \
     \sum_{j = 1}^l \left(
       \sum_{\mu \in \Map{h}{l}}
       \nabla^{\mu^{-1}(1)}_{\xi_{\bullet}} A_1
       \cdots
       \nabla_{\xi_0} \nabla^{\mu^{-1}(j)}_{\xi_{\bullet}} A_j
       \cdots
       \nabla^{\mu^{-1}(l)}_{\xi_{\bullet}} A_l
  \right).
  \end{multline*}
  Thanks to the tensorial nature of the multi-covariant derivative, we can
  assume $\nabla_{\xi_0} \xi_p \left( x \right) = 0$, $p = 1, \ldots, k$ at
  some arbitrary point $x$. Then
  \begin{multline*}
    \left( \bigotimes_{p = 0}^h \xi_p \right)\neg\nabla^{h + 1}
    \left( \prod_{j = 1}^l A_j \right) \\
    =\
    \sum_{j = 1}^l\left(
      \sum_{\mu \in \Map{h}{l}}
      \nabla^{\mu^{-1}(1)}_{\xi_{\bullet}} A_1
      \cdots
      \nabla^{\{0\} \cup \mu^{-1}(j)}_{\xi_{\bullet}} A_j
      \cdots
      \nabla^{\mu^{-1}(l)}_{\xi_{\bullet}} A_l
    \right).
  \end{multline*}
  The conclusion follows from the observation that
  \begin{equation}
    \Map{\{0, \ldots, h\}}{l} =
    \{\mu_j \mid \mu \in \Map{h}{l}, j \in [l]\}\,,
  \end{equation}
  where the map $\mu_j : \{0, \ldots, h\} \longrightarrow [l]$ is defined as
  by $\mu_j(0) = j$ and
  $\mu_j(i) = \mu(i)$ for $i\neq 0$.
\end{proof}

\begin{corollary}
  If $\xi$ is a vector field over $X$ then
  \begin{equation*}
    \xi^{\otimes h} \neg \frac{1}{h!} \nabla^h
    \left( \prod_{j = 1}^l A_j \right)
    \ = \
    \sum_{H\in\N^l(h)}\ \prod_{j = 1}^l
    \left( \xi^{\otimes h_j} \neg \frac{1}{h_j !} \nabla^{h_j} A_j \right)\,.
  \end{equation*}
\end{corollary}

We infer the inequality with respect to the pointwise max norm on multilinear
forms
\begin{equation*}
  \frac{1}{h!}  \left\| \nabla^h  \left( \prod_{j = 1}^l A_j \right) \right\|
  \ \leqslant\
  \sum_{H\in\N^l(h)}\ \prod_{j = 1}^l\ \frac{1}{h_j!}
  \| \nabla^{h_j} A_j \|\,.
\end{equation*}
The previous pointwise inequality leads to the global estimate
\begin{equation}
  \label{estimProd}
  \frac{1}{h!} \left\| \nabla^h
    \left( \prod_{j = 1}^l A_j \right)
  \right\|_r
  \ \leqslant\
  C^{l - 1}_r \sum_{H\in\N^l(h)}\ \prod_{j = 1}^l\
  \frac{1}{h_j !}  \| \nabla^{h_j} A_j \|_r\,.
\end{equation}
\section{Proof of the theorem \ref{MainTeo}}

In this final section, we need to consider trees on the union of two families
of vector fields $\eta_{\bullet}, \xi_{\bullet}$. This means that our label
set for the trees will be a subset of two copies of $\N_{\geq 1}$ together
with the usual empty root $\circ$. To distinguish the two copies, we write
them $\N_{\geq 1}'=\{1', 2',\dots\}$ and $\N_{\geq 1}=\{1, 2,\dots\}$. Recall
that the \emph{ordered sum} $S+T$ of two totally ordered sets $(S,
\leq_S)$ and $(T, \leq_T)$ is the disjoint union $S+T := S\sqcup T$ together
with the order $\leq_{S+T}$ which keeps the relative order of the sets and
such that all the elements of $S$ are smaller than the element of $T$. Formally,
\begin{equation*}
  x\leq_{S+T} y\quad\Longleftrightarrow\quad
  \begin{cases}
    \text{$x\leq_S y$ if $x,y\in S$ or}\\
    \text{$x\leq_T y$ if $x,y\in T$ or}\\
    \text{$x\in S$ and $y\in T$}.\\
  \end{cases}
\end{equation*}
\begin{definition}
  Let $S'\subset\N_{\geq 1}'$ and $S\subset\N_{\geq 1}\cup\{\circ\}$. Let
  $\mu\in\Map{S'}{S}$ and $F$ a forest on $S$. We denote $\mu\cup F$ the forest
  on $S' + S$ with father function $f$ defined by
  $f(i') = \mu(i')$ for $i'\in S'$ and $f(i)=F(i)$ for $i\in S$.
\end{definition}
We remark that the roots of $\mu\cup F$ are the same as the roots of $F$ so
that if $F$ is actually a tree $T$ then $\mu\cup T$ is also a tree.

\begin{example}
  Let $S=\{1,3,4,5,8,9,\circ\}$ and
  $S'=\{1',2',4',6',7',9'\}$. Consider the map $\mu=\left\{
    \begin{array}{@{\,}c@{\,}c@{\,}c@{\,}c@{\,}c@{\,}c}
      1'& 2'& 4'& 6'& 7'& 9'\\
      8 & 8 & 3 & \circ & 9 & 8\\
    \end{array}\right.$.
  The picture below show some tree $T$ on $S$ together with the associated
  $\mu\cup T$ tree:
  \[
  T =
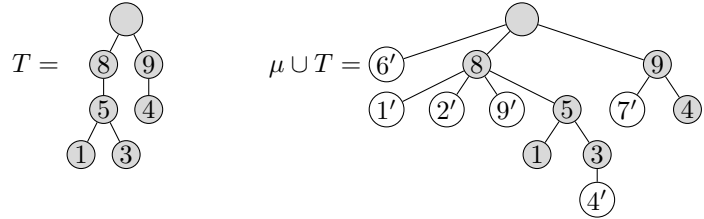
  \begin{tikzpicture}[baseline=-20, scale=0.4, inner sep=0.3mm]
    \node[circle, draw, fill=gray!30] {\phantom{$A$}}
    child {node[circle, draw, fill=gray!30] {$8$}
      child {node[circle, draw, fill=gray!30] {$5$}
        child {node[circle, draw, fill=gray!30] {$1$}}
        child {node[circle, draw, fill=gray!30] {$3$}}
      }
    }
    child {node[circle, draw, fill=gray!30] {$9$}
      child {node[circle, draw, fill=gray!30] {$4$}}
    } ;
  \end{tikzpicture}\qquad\qquad
  \mu\cup T =
  \begin{tikzpicture}[baseline=-20, scale=0.4, inner sep=0.3mm, level
    1/.style={sibling distance=3cm},level 2/.style={sibling distance=2cm}]
    \node[circle, draw, fill=gray!30] {\phantom{$A$}}
    child {node[circle, draw] {$6'$}}
    child {node[circle, draw, fill=gray!30] {$8$}
      child {node[circle, draw] {$1'$}}
      child      {node[circle, draw] {$2'$}}
      child {node[circle, draw] {$9'$}}
      child      {node[circle, draw, fill=gray!30] {$5$}
        child {node[circle, draw, fill=gray!30] {$1$}}
        child        {node[circle, draw, fill=gray!30] {$3$}
          child {node[circle, draw] {$4'$}}
        }
      }
    }
    child [missing] {}
    child {node[circle, draw, fill=gray!30] {$9$}
      child {node[circle, draw] {$7'$}}
      child {node[circle, draw, fill=gray!30] {$4$}}
    } ;
  \end{tikzpicture}
  \]
  Expanding the associated nested derivative as in Definition~\ref{def-tree} gives:
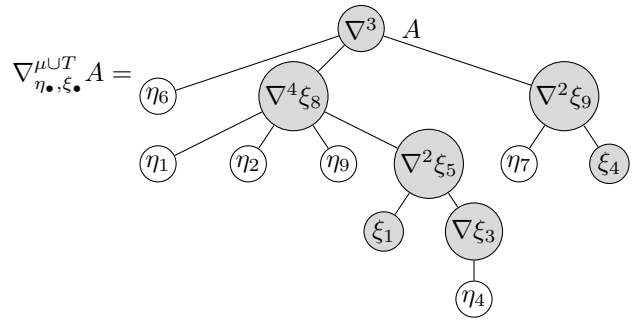
  \[\nabla_{\eta_{\bullet}, \xi_{\bullet}}^{\mu\cup T} A =
  \begin{tikzpicture}[baseline=-20, scale=0.6, inner sep=0.3mm, level
    1/.style={sibling distance=3cm},level 2/.style={sibling distance=2cm}]
    \node[right=0.5cm] {$A$};
    \node[circle, draw, fill=gray!30] {$\nabla^3$}
    child {node[circle, draw] {$\eta_6$}}
    child {node[circle, draw, fill=gray!30] {$\nabla^4{\xi_8}$}
      child {node[circle, draw] {$\eta_1$}}
      child {node[circle, draw] {$\eta_2$}}
      child {node[circle, draw] {$\eta_9$}}
      child {node[circle, draw, fill=gray!30] {$\nabla^2{\xi_5}$}
        child {node[circle, draw, fill=gray!30] {$\xi_1$}}
        child {node[circle, draw, fill=gray!30] {$\nabla{\xi_3}$}
          child {node[circle, draw] {$\eta_4$}}
        }
      }
    }
    child [missing] {}
    child {node[circle, draw, fill=gray!30] {$\nabla^2{\xi_9}$}
      child {node[circle, draw] {$\eta_7$}}
      child {node[circle, draw, fill=gray!30] {$\xi_4$}}
    };
  \end{tikzpicture}
  \]
\end{example}
We remark that for any node $\nu$ of $T$
\begin{equation*}
  \nchild[\mu\cup T]{\nu} =
  \left| \mu^{- 1} \left( \nu \right) \right| + \nchild[T]{\nu}\,.
\end{equation*}
Moreover according to Definition~\ref{def-tree},
$\nabla_{\eta_{\bullet}, \xi_{\bullet}}^{\mu\cup T} A$ is given by
\begin{equation*}
  \nabla_{\eta_{\bullet}, \xi_{\bullet}}^{\mu\cup T} A
  \ = \
  \left[
    \left( \bigotimes_{p \in \mu^{- 1} \left( \emptyroot \right)} \eta_p \right)
    \otimes
    \left( \bigotimes_{\nu \in\child{\emptyroot}}
      \nabla^{\mu\cup T_{\nu}}_{\eta_{\bullet}, \xi_{\bullet}} \xi_{\nu}
    \right)
  \right] \neg
  \nabla^{\nchild[\mu\cup T]{\emptyroot}} A\,,
\end{equation*}
with
\begin{equation*}
  \nabla^{\mu\cup T_{\nu}}_{\eta_{\bullet}, \xi_{\bullet}} \xi_{\nu}
  \ = \
  \left[
    \left( \bigotimes_{p \in \mu^{- 1} \left( \nu \right)} \eta_p \right)
    \otimes
    \left( \bigotimes_{n \in \child{\nu}}
      \nabla^{\mu\cup T_n}_{\eta_{\bullet}, \xi_{\bullet}} \xi_{n}
    \right)
  \right] \neg
  \nabla^{\nchild[\mu\cup T]{\nu}} \xi_{\nu}\,,
\end{equation*}
and so on, with $\nabla^{\mu\cup\emptyset}_{\eta_{\bullet}, \xi_{\bullet}}
\assign \mathbbm{I}$ and the abuse of notation
$\mu\cup H := \mu|_{\mu^{-1}(H)}\cup H$ for any subtree $H$ of $T$.
\bigskip

\begin{proof}[Proof of theorem~\ref{MainTeo}]
We apply recursively Lemma~\ref{abstrLeibnitz} in case~\ref{case-prod-tree}
with
\begin{itemize}
\item $A_1:=\nabla^{l_{\rho(T)}}A$,
\item $A_i:=\nabla^{T_{\nu}}_{\xi_{\bullet}} \xi_{\nu}$, if $i>1$, where $\nu$
  is the $i$-th child of $\rho(T)$.
\end{itemize}
We get
\begin{equation}\label{diff-mult-tree}
  \nabla^{S'}_{\eta_\bullet}
  \left(\nabla_{\xi_\bullet}^T A\right)
  \ \equiv\
  \left( \bigotimes_{p \in S'} \eta_p \right) \neg \nabla^{|S'|}
  \left(\nabla_{\xi_\bullet}^T A\right)
  \ =\
  \sum_{\mu \in \Map{S'}{T}}
  \nabla_{\eta_{\bullet}, \xi_{\bullet}}^{\mu\cup T} A\,.
\end{equation}
By formula (\ref{multipLIEforest}) and linearity, we get
\begin{multline*}
  \left( \bigotimes_{p = 1}^h \eta_p \right) \neg \nabla^h \left[\left( \prod_{j =
        1}^k L_{\xi_j} \right) A \right]
  \\
  =
  \sum_{F \in\Forests_{k^\circ}} (-1)^{l_F - 1}
  \left( \bigotimes_{p = 1}^h \eta_p \right) \neg \nabla^h
  \left[\left( \prod_{T \in
  F^{\dagger}} \ad \left( \nabla_{\xi_{\bullet}}^{T} \nabla
  \xi_{\Root{T}} \right) \right) \nabla_{\xi_{\bullet}}^{F_\circ} A\right]\,.
\end{multline*}
We use now Lemma~\ref{abstrLeibnitz} in case~\ref{case-prod-comp},
writing $F=\{T_1,\dots,T_{l_F-1},T_{l_F}=F_{\circ}\}$ in the increasing order provided
by the labels of the roots of $F$, from the left to the right
with
\begin{itemize}
\item $A_i:=\ad \left( \nabla_{\xi_{\bullet}}^{T_i} \nabla
    \xi_{\Root{T_i}}\right)$
  for $i=1,\dots,T_{l_F-1}$\,,
\item $A_{l_F}:=\nabla_{\xi_{\bullet}}^{F_\circ}A$\,.
\end{itemize}
One obtains, for a fixed forest $F$,
\begin{align*}
  G_F&\ \assign\
  \left( \bigotimes_{p = 1}^h \eta_p \right) \neg \nabla^h
  \left[\left(
      \prod_{T \in F^{\dagger}}
      \ad \left( \nabla_{\xi_{\bullet}}^{T} \nabla \xi_{\Root{T}} \right)
    \right) \nabla_{\xi_{\bullet}}^{F_\circ} A\right]\\[3mm]
  &\ =\
  \sum_{\mu\in \Map{h}{l_F}}\left[
    \prod_{j=1}^{l_F-1}
    \ad \left(
      \nabla^{\mu^{-1}(j)}_{\eta_\bullet}
      \nabla_{\xi_{\bullet}}^{T_j} \nabla \xi_{\Root{T_j}} \right)\right]
  \nabla^{\mu^{-1}(l_F)}_{\eta_\bullet} \nabla_{\xi_{\bullet}}^{F_\circ} A\,.
  \intertext{Applying Equation~\ref{diff-mult-tree} to each $T_j$, we obtain}
  G_F&\ =\
  \sum_{\mu\in \Map{h}{l_F}}\left[
    \prod_{j=1}^{l_F-1}
    \ad \left(
      \sum_{\beta_j\in \Map{\mu^{-1}(j)}{T_j}}
      \nabla_{\eta_{\bullet}, \xi_{\bullet}}^{\beta_j\cup T_j}
      \nabla \xi_{\Root{T_j}}
    \right)\right]\times \\
  &
  \times \left(
    \sum_{\beta_{l_F}\in \Map{\mu^{-1}(l_F)}{F_\circ}}
    \nabla_{\eta_{\bullet}, \xi_{\bullet}}^{\beta_{l_F}\cup F_\circ}A
  \right)\,.
\end{align*}
We now recombine the maps $\mu, (\beta_j)_{j=1..l_F}$ into a single map
$\alpha\in\Map{h}{k^\circ}$ by setting $\alpha(p) := \beta_{\mu(p)}(p)$. Each
$\alpha\in\Map{h}{k^\circ}$ is obtained exactly once from a pair
$(\mu,(\beta_j)_{j=1..l_F})$. Then
\begin{multline*}
  \left( \bigotimes_{p = 1}^h \eta_p \right) \neg
  \nabla^h \left[\left( \prod_{j = 1}^k L_{\xi_j} \right) A \right]
  \\
  = \sum_{\substack{
        F \in \Forests_{k^\circ}\\
        \alpha \in \Map{h}{k^\circ}
      }}
    (-1)^{l_F - 1}
    \left[
      \prod_{T \in F^{\dagger}}
      \ad \left( \nabla_{\eta_{\bullet}, \xi_{\bullet}}^{\alpha\cup T} \nabla
      \xi_{\Root{T}} \right)
  \right]
  \nabla_{\eta_{\bullet}, \xi_{\bullet}}^{\alpha\cup F_\circ} A\,,
\end{multline*}
where we again used the abuse of notation $\alpha\cup T :=
\alpha|_{\alpha^{-1}(T)}\cup T$ for any subtree $T$ of $F$.

In the case $\eta = \eta_p$ for all $p$, we obtain as for the inequality
(\ref{estimProd})
\begin{align}
  &  \frac{1}{h!}  \left\| \nabla^h \left( \prod_{j = 1}^k L_{\xi_j}
    \right) A \right\|_r\nonumber
  \\[3mm]
  &\  \leqslant C^k_r \sum_{\substack{
    F \in \Forests_{k^\circ}\\
    H \in \N^{k + 1}(h)
  }} 2^{l_F - 1}\ \times\nonumber
  \\
  &\ \times
  \left( \frac{1}{h_{k + 1} !}  \left\| \nabla^{h_{k
  + 1} + \nchild{\circ}} A \right\|_r \cdot \prod_{\nu \in F^{\ast}_\circ} \frac{1}{h_{\nu} !} \| \nabla^{h_{\nu} + \nchild{\nu}}
  \xi_{\nu} \|_r \right) \times\nonumber\\
  &\ \times \prod_{T \in F^{\dagger}} \left( \frac{1}{h_{\Root{T}} !}
  \left\| \nabla^{h_{\Root{T}} + 1 + \nchild{\Root{T}}} \xi_{\Root{T}}
  \right\|_r \cdot \prod_{\nu \in T^{\ast}} \frac{1}{h_{\nu} !} \|
  \nabla^{h_{\nu} + \nchild{\nu}} \xi_{\nu} \|_r \right)\,.\label{estmHCovLieJ}
\end{align}
For any fixed $H \in\N^{k + 1}(h)$ we consider the terms of the sum in
(\ref{estmHCovLieJ}) and we set $B^l \assign \frac{1}{h_{k + 1} !}  \|
\nabla^{h_{k + 1} + l} A\|_r$ and $X^{l}_{\nu} \assign
\frac{1}{h_{\nu} !}  \| \nabla^{h_{\nu} + l} \xi_{\nu}
\|_r$. Then the estimate in the statement of theorem \ref{MainTeo} follows
from theorem \ref{TheoLieTree}.
\end{proof}

\section*{Aknowledgment}

This research was driven by computer exploration, using the open-source math-
ematical software Sagemath~\cite{Sagemath} and its algebraic combinatorics
features developed by the Sage-Combinat community~\cite{SC}.

\vspace{1cm}
\noindent
Florent Hivert
\\
Laboratoire de Recherche en Informatique,
Universit\'{e} Paris Sud,
\\
B\^{a}timent 650 F91405 Orsay, France
\\
E-mail: \texttt{florent.hivert@lri.fr}

\vspace{0.5cm}
\noindent
Nefton Pali
\\
Universit\'{e} Paris Sud, D\'epartement de Math\'ematiques
\\
B\^{a}timent 307 F91405 Orsay, France
\\
E-mail: \texttt{nefton.pali@math.u-psud.fr}

\end{document}